\documentclass{amsart} 
\usepackage{amsmath}
\usepackage{mathrsfs}
\usepackage{amssymb}

\newcommand{\map}[1]{\xrightarrow{#1}}

\input xy 
\xyoption{all}
\begin{document}

\author{Benjamin Howard}
\thanks{This research was  supported by an NSF 
postdoctoral fellowship.}
\address{Dept. of Mathematics, Harvard University, 1 Oxford St., Cambridge, MA. 02138.}

\title{Special cohomology classes for modular Galois representations}

\maketitle

%\subjclass[2000]{11F80, 11G18}

\begin{abstract}
Building on ideas of Vatsal \cite{vatsal02}, Cornut \cite{cornut02} proved 
a conjecture of Mazur asserting the generic nonvanishing of Heegner
points on an elliptic curve $E_{/\mathbb{Q}}$ as one ascends the 
anticyclotomic  $\mathbb{Z}_p$-extension of a quadratic imaginary 
extension $K/\mathbb{Q}$.
In the present article Cornut's result is extended by replacing the elliptic
curve $E$ with the Galois cohomology of Deligne's $2$-dimensional 
$\ell$-adic representation attached to a modular form of weight $2k>2$,
and replacing the family of Heegner points with an analogous family of special 
cohomology classes.
\end{abstract}

\theoremstyle{plain} 

\newtheorem{Thm}{Theorem}[subsection]
\newtheorem{Prop}[Thm]{Proposition} 
\newtheorem{Lem}[Thm]{Lemma}
\newtheorem{Cor}[Thm]{Corollary} 
\newtheorem{BigTheorem}{Theorem}

\theoremstyle{definition} 
\newtheorem{Def}[Thm]{Definition}
\newtheorem{Hyp}[Thm]{Hypothesis} 
\newtheorem{Con}[Thm]{Construction}

\theoremstyle{remark}
\newtheorem{Rem}[Thm]{Remark}
\newtheorem{Ques}[Thm]{Question}

\renewcommand{\labelenumi}{(\alph{enumi})}
\renewcommand{\theBigTheorem}{\Alph{BigTheorem}}
\setcounter{section}{-1}

%%%%%%%%%%%%%%%%%%%%%%%%%%%%%%%%%%%%%%%%%%%%%%%%%%%%%%%%%%%%%%%%%%%%%

\section{Introduction}

%%%%%%%%%%%%%%%%%%%%%%%%%%%%%%%%%%%%%%%%%%%%%%%%%%%%%%%%%%%%%%%%%%%%%

\subsection{Statement of the main result}
\label{intro}

Let $f\in S_{2k}(\Gamma_0(N),\mathbb{C})$ be a normalized newform of weight
$2k> 2$ and level $N\ge 4$.  Fix
a rational prime $\ell$ and embeddings of algebraic closures 
$\mathbb{Q}^\mathrm{al}\hookrightarrow \mathbb{Q}_\ell^\mathrm{al}$, 
$\mathbb{Q}^\mathrm{al}\hookrightarrow \mathbb{C}$.
Let $\Phi\subset\mathbb{Q}_\ell^\mathrm{al}$ be a 
finite extension of $\mathbb{Q}_\ell$
containing all Fourier coefficients of $f$ and let 
$W_f$ be the $2$-dimensional $\Phi$ vector space with 
$\mathrm{Gal}(\mathbb{Q}^\mathrm{al}/\mathbb{Q})$-action 
constructed by Deligne \cite{deligne},
so that the geometric Frobenius
of a prime $q\nmid \ell N$ acts on $W_f$ with characteristic polynomial
$X^2-a_q(f)X+q^{2k-1}$.  Let $K$ be a quadratic imaginary field
satisfying the \emph{Heegner hypothesis} that all prime divisors of 
$N$ are split in $K$, fix a prime $p\nmid N\cdot \mathrm{disc}(K)$,
let $H[p^s]$ denote the ring class field of $K$ of conductor $p^s$,
set $H[p^\infty]=\cup_s H[p^s]$, and define 
$\mathcal{G}=\mathrm{Gal}(H[p^\infty]/K)$.  
The torsion subgroup 
$G_0\subset \mathcal{G}$ satisfies $\mathcal{G}/G_0\cong\mathbb{Z}_p$.

In \S \ref{S: Heegner cohomology} we define for every $s\ge 0$
a subspace
\begin{displaymath}
\mathrm{Heeg}_s(f)\subset H^1(H[p^s],W_f(k)).
\end{displaymath}
This subspace is the higher weight analogue of the subspace generated 
by the Kummer images of Heegner points in the case $k=1$, in which case
\begin{displaymath}
W_f(1)\cong \mathrm{Ta}_\ell(A_f)\otimes\mathbb{Q}_\ell
\end{displaymath} 
for $A_f$  the modular abelian
variety attached to $f$ by Eichler-Shimura theory.  Such higher
weight Heegner objects have been studied earlier by
Brylinski \cite{bryl}, Nekov{\'a}{\v{r}} \cite{nek:euler,nek:gz},
and Zhang \cite{zhang97}, and our construction of $\mathrm{Heeg}_s(f)$
follows  Nekov{\'a}{\v{r}}'s \cite{nek:gz} construction very closely.
The main result (Theorem \ref{big money}) extends the results of
Cornut \cite{cornut02} and Vatsal \cite{vatsal02}
from the case $k=1$, and is as follows:

\begin{BigTheorem}
\label{BT}
Fix a character $\chi: G_0\map{}\Phi^\times$ and let
$$
\pi_\chi=\sum_{\sigma\in G_0}\chi(\sigma)\sigma
\in \Phi[G_0].
$$
Suppose $\ell\nmid p\cdot N\cdot \varphi(N)\cdot \mathrm{disc}(K)\cdot (2k-2)!$
($\varphi$ is Euler's function).  
As $s\to\infty$ the $\Phi$-dimension
of $\pi_\chi\mathrm{Heeg}_s(f)$ grows without bound.
\end{BigTheorem}

Let $X(N)_{/\mathbb{Q}}$ be the usual (geometrically disconnected) 
moduli space of
generalized elliptic curves over $\mathbb{Q}$ with full level $N$ 
structure, and
let $\mathcal{V}_{/\mathbb{Q}}\map{}X(N)_{/\mathbb{Q}}$ 
be the Kuga-Sato variety considered in 
\cite{scholl2,scholl1}.  Thus $\mathcal{V}_{/\mathbb{Q}}$ is a 
desingularization of the 
$(2k-2)$-fold fiber product over $X(N)_{/\mathbb{Q}}$ 
of the universal generalized elliptic curve. 
By work of Scholl \cite{scholl1}, Deligne's
$\ell$-adic representation $W_f$ occurs as a summand 
of $H^{2k-1}(\mathcal{V}_{/\mathbb{Q}^\mathrm{al}},\mathbb{Q}_\ell)$.
Combining this with the $\ell$-adic Abel-Jacobi map of \cite{nek:AJ} 
yields a map \cite[\S 0.3]{nek:gz}
\begin{displaymath}
\Psi_f:\mathrm{CH}^k_0(\mathcal{V}_{/F}) \map{}
H^{2k}(\mathcal{V}_{/F},\mathbb{Q}_\ell(k))  \map{} H^1(F, W_f(k))
\end{displaymath}
for any number field $F$, where $\mathrm{CH}_0^k$ denotes the Chow group of 
homologically trivial cycles of codimension $k$, modulo rational equivalence.
 Nekov{\'a}{\v{r}} \cite{nek:gz,nek:AJ} shows that the image of 
$\Psi_f$ is contained
in the Bloch-Kato Selmer group 
\begin{displaymath}
\mathrm{Sel}(F,W_f(k))\subset H^1(F,W_f(k)).
\end{displaymath}
Taking $F=H[p^s]$, the subspace 
$\mathrm{Heeg}_s(f)$ lies in the image of $\Psi_f$.  

As in \cite{vatsal02} we may write $H[p^\infty]$ as the compositum
of linearly disjoint (over $K$) fields $F$ and $K_\infty$ where $F/K$
is tamely ramified at $p$ with Galois group $G_0$,
 and $K_\infty/K$ is the anticylotomic $\mathbb{Z}_p$-extension.
By Theorem \ref{BT} (and under the hypotheses of that theorem),
the dimension of the $\chi$-component of  
$\mathrm{Sel}(\mathbb{Q}^\mathrm{al}/H[p^s],W_f(k))$
 grows without bound.
This provides some evidence for the standard conjecture predicting 
that for each character $\chi$ of $G_0$
\begin{equation}\label{BK}
\mathrm{dim}_\Phi\ \pi_\chi \mathrm{Sel}(H[p^s],W_f(k))=
\mathrm{ord}_{s=k}
\prod_{\psi}L(f\otimes K,\chi^{-1}\psi,s)
\end{equation}
where the product is over all characters $\psi$ of 
$\mathrm{Gal}(K_\infty/K)$ of conductor $\le p^s$
and $L(f\otimes K,\chi^{-1}\psi,s)$ is the twisted $L$-function
defined as in \cite[\S 0.5]{nek:gz}.  Indeed, the Heegner 
hypothesis and the functional
equation force $L(f\otimes K,\chi^{-1}\psi,k)=0$ for each such $\psi$, 
and so the right hand is $\ge p^s$.
One might hope to extend Kolyvagin's theory of Euler systems so as
to prove that the left hand side is $p^s+O(1)$.
Work of Nekov{\'a}{\v{r}} \cite{nek:euler} and of
Bertolini and Darmon \cite{bertdar} give evidence that this is accessible.

It is conjectured that the kernel of $\Psi_f$ is independent of
the choice of prime $\ell$.  A proof would allow one to remove 
the undesirable hypothesis that $\ell\not=p$ in Theorem \ref{BT},
leading to higher weight generalizations of the Iwasawa theoretic
results of \cite{bert,me,pr}.  It seems difficult to adapt the
methods of the present article to treat the (most interesting)
case $\ell=p$; instead
that case is treated in the
forthcoming work \cite{me:hida} using a completely different construction of 
Heegner cohomology classes in $H^1(H[p^s], W_f(k))$.
The constructions and results of 
\cite{me:hida} hold only for $\ell=p$ and $f$ ordinary at $p$, 
but allow modular forms of odd weight (which seem
inaccessible using the methods of the present article).

Zhang \cite{zhang97} has proved a higher weight
form of the Gross-Zagier theorem relating the height pairings
of certain Heegner cycles in $\mathrm{CH}^k_0(\mathcal{V}_{/H[1]})$ to the 
derivatives $L'(f\otimes K,\chi^{-1}\psi,k)$ for characters $\psi$ 
of trivial conductor.
The images of these Heegner cycles under $\Psi_f$ generate
our $\mathrm{Heeg}_0(f)$, and thus Theorem \ref{BT} would 
yield nonvanishing results for $L'(f\otimes K,\chi^{-1}\psi,k)$
if Zhang's formula were extended to ramified characters, and (harder) if
one knew the nondegeneracy of the height pairing on the Chow group 
$\mathrm{CH}_0^k(\mathcal{V})$.

%%%%%%%%%%%%%%%%%%%%%%%%%%%%%%%%%%%%%%%%%%%%%%%%%%%%%%%%%%%%%%%%%%%%

\subsection{Notation and conventions}
\label{notation}

%%%%%%%%%%%%%%%%%%%%%%%%%%%%%%%%%%%%%%%%%%%%%%%%%%%%%%%%%%%%%%%%%%%%

Throughout this article we use $k$, $N$, and $M$ to denote 
positive integers with $k> 1$, $N\ge 4$, $M$ squarefree, and $(M,N)=1$. 
We will be ultimately be concerned with the case $M=1$, but must allow more
general $M$ for technical reasons ($M$ will eventually be a 
divisor of $\mathrm{disc}(K)$). We frequently abbreviate $\mathbf{N}=NM$.
The letters $\ell$ and  $p$ denote rational primes with 
$(\ell p, \mathbf{N})=1$.  
We allow $\ell=p$ unless stated otherwise (more precisely,
we allow $\ell=p$ except in Sections \ref{S: reduction of the family}
and \ref{S: money}).

The letter $\Lambda$ always denotes a $\mathbb{Z}_\ell$-algebra.
If $S$ is a scheme on which $\mathbf{N}$ is invertible we let 
$Y_0(\mathbf{N})_{/S}$ (resp. $Y_1(N,M)_{/S}$)
be the coarse (resp. fine) moduli space of elliptic curves with
$\Gamma_0(\mathbf{N})$ level structure (resp. 
$\Gamma_1(N,M)=\Gamma_1(N)\cap \Gamma_0(M)$
level structure).
If $M=1$ we omit it from the notation, and we sometimes omit $S$
if it is clear from the context.  
For a congruence subgroup $\Gamma\subset\mathrm{SL}_2(\mathbb{Z})$ we
will sometimes refer to a pair $(E,x)$ consisting of 
an elliptic curve $E$ together with a $\Gamma$ level
structure $x$ on $E$ simply as a $\Gamma$ structure. 
Set $\Delta=(\mathbb{Z}/N\mathbb{Z})^\times$ and let 
$\varphi$ be Euler's function.

%%%%%%%%%%%%%%%%%%%%%%%%%%%%%%%%%%%%%%%%%%%%%%%%%%%%%%%%%%%%%%%%%%
%%%%%%%%%%%%%%%%%%%%%%%%%%%%%%%%%%%%%%%%%%%%%%%%%%%%%%%%%%%%%%%%%%

\section{Augmented elliptic curves}
\label{S: kummer}

%%%%%%%%%%%%%%%%%%%%%%%%%%%%%%%%%%%%%%%%%%%%%%%%%%%%%%%%%%%%%%%%%%
%%%%%%%%%%%%%%%%%%%%%%%%%%%%%%%%%%%%%%%%%%%%%%%%%%%%%%%%%%%%%%%%%%

Throughout \S \ref{S: kummer},
$\Lambda=\mathbb{Z}/m\mathbb{Z}$ for some fixed $\ell$-power $m$, and 
$S=\mathrm{Spec}(F)$ for $F$ a perfect field of characteristic prime to 
$\ell \mathbf{N}$ with algebraic closure $F^\mathrm{al}$.  

\subsection{Sheaves}
\label{sheaves}

If $L/F$ is an algebraic extension,
let $\pi^\mathrm{univ}:E^\mathrm{univ}\map{}Y_1(N)_{/L}$ 
be the universal elliptic curve, and define
a locally constant constructible sheaf on $Y_1(N)_{/L}$
\begin{equation}\label{sheaf def}
\mathcal{L}_{\Lambda}=\mathrm{Sym}^{2k-2}
(R^1\pi_*^\mathrm{univ} \underline{\Lambda}).
\end{equation}
The formation of this sheaf is compatible with
base change in $L$, by the proper base change theorem.
There are isomorphisms of \'etale sheaves on $Y_1(N)_{/L}$
\begin{displaymath}
R^1\pi_*^\mathrm{univ}\underline{\mu}_m\cong 
\underline{\mathrm{Hom}}(\underline{E}^\mathrm{univ}[m],\underline{\mu}_m)
\cong \underline{E}^\mathrm{univ}[m]
\end{displaymath}
where $\underline{E}^\mathrm{univ}[m]$ is the \'etale sheaf on $Y_1(N)_{/L}$
associated to the group scheme $E^\mathrm{univ}[m]$ and 
$\underline{\mathrm{Hom}}$
is sheaf $\mathrm{Hom}$.
Taking symmetric powers, there is a canonical isomorphism 
\begin{equation}\label{nice sheaf iso}
\mathcal{L}_\Lambda(2k-2)\cong \mathrm{Sym}^{2k-2}
(\underline{E}^\mathrm{univ}[m]).
\end{equation}  

If we let $Y_{/L}$ be a connected component of the open modular curve
parameterizing elliptic curves over $L$
with $\Gamma_1(N)\cap \Gamma(m)$ level structure
and fix a geometric point 
$\bar{z}\map{}Y_1(N)_{/L}$,
then the forgetful covering map $Y_{/L}\map{}Y_1(N)_{/L}$ 
cuts out a quotient of the  fundamental group
$\pi_1=\pi_1(Y_1(N)_{/L},\bar{z})$.
The group of $Y_{/L}$-valued $m$-torsion points of the 
universal elliptic curve over $Y_{/L}$ is canonically isomorphic to
$\Lambda^2$ (via the universal $\Gamma(m)$ level structure), and
the action of $\pi_1$ on this group identifies the aforementioned quotient 
with a subgroup of $\mathrm{GL}_2(\Lambda)$
containing $\mathrm{SL}_2(\Lambda)$.  We thus obtain an action 
of $\pi_1$ on $\Lambda^2$ and so also on $\mathrm{Sym}^{2k-2}\Lambda^2$.
It is immediate from (\ref{nice sheaf iso}) that the 
locally constant sheaf associated to this action is isomorphic to
$\mathcal{L}_\Lambda(2k-2)$. From the discussion following
\cite[\S 2 Lemma 2]{diamond-taylor}
we see that there is a  perfect symmetric pairing of \'etale sheaves
\begin{equation}\label{nice sheaf pairing}
\mathcal{L}_\Lambda(k-1)\otimes\mathcal{L}_\Lambda(k-1)\map{}
\underline{\Lambda}.
\end{equation}

For any \'etale sheaf $\mathcal{F}$ on $Y_1(N)_{/L}$ define
\begin{equation}\label{cuspidal cohomology}
\tilde{H}^*(Y_1(N)_{/L},\mathcal{F})=
\mathrm{Image}\big(
H^*_c(Y_1(N)_{/L}, \mathcal{F})
\map{}
H^*(Y_1(N)_{/L}, \mathcal{F})
\big).
\end{equation}

%%%%%%%%%%%%%%%%%%%%%%%%%%%%%%%%%%%%%%%%%%%%%%%%%%%%%%%%55

\subsection{Augmentations}

%%%%%%%%%%%%%%%%%%%%%%%%%%%%%%%%%%%%%%%%%%%%%%%%%%%%%%%%%%

Let $L/F$ be an algebraic extension and let 
$\Gamma$ be any one of $\Gamma_0(N)$, $\Gamma_0(\mathbf{N})$,
$\Gamma_1(N)$, or $\Gamma_1(N,M)$.
If $E$ is an elliptic curve over $L$, define a 
$\mathrm{Gal}(F^\mathrm{al}/L)$-module
\begin{displaymath}
\mathcal{A}_\Lambda(E)=(\mathrm{Sym}^{2k-2}E[m])(1-k)
\end{displaymath}
(where $E[m]=E(F^\mathrm{al})[m]$) and set
$
\mathcal{A}^\circ_\Lambda(E)=
\mathcal{A}_\Lambda(E)^{\mathrm{Gal}(F^\mathrm{al}/L)}.
$
Note that $\mathcal{A}_\Lambda$ and $\mathcal{A}^\circ_\Lambda$ are
naturally covariant functors on the category of elliptic curves over $L$.
The construction of $\mathcal{A}_\Lambda(E)$ depends on the
embedding $L\hookrightarrow F^\mathrm{al}$, 
but that of $\mathcal{A}_\Lambda^\circ(E)$
does not, in the sense that the $\Lambda$-modules defined by two
different choices are canonically isomorphic.

\begin{Def}
By a \emph{$\Lambda$-augmented $\Gamma$ structure} over an algebraic
extension $L/F$ we mean a triple $(E,x,\Theta)$ in which 
$(E,x)$ is an elliptic curve with  $\Gamma$ structure
over $L$ and  $\Theta\in\mathcal{A}_\Lambda^\circ(E)$.
\end{Def}

Two  $\Lambda$-augmented $\Gamma$ structures over $L$,
$(E_0,x_0,\Theta_0)$ and $(E_1,x_1,\Theta_1)$, are \emph{isomorphic}
if there is an isomorphism (over $L$) of elliptic curves $\phi:E_0\map{}E_1$
such such that $\phi$ identifies $x_0$ with $x_1$ and 
$\phi(\Theta_0)=\Theta_1$.
If $(E,x,\Theta)$ is a $\Lambda$-augmented $\Gamma$ structure
over $F^\mathrm{al}$ and
$\sigma$ is an automorphism of $F^\mathrm{al}$, there is an evident 
notion of the \emph{conjugate} $\Lambda$-augmented $\Gamma$ structure
$(E,x,\Theta)^\sigma=(E^\sigma,x^\sigma,\Theta^\sigma)$.

\begin{Def}\label{moduli definition}
Given a $\Lambda$-augmented $\Gamma$ structure 
$(E,x,\Theta)$ over $F^\mathrm{al}$
the \emph{field of moduli}, $L$, of $(E,x,\Theta)$ is the 
extension of $F$ characterized by the property that
$\sigma\in\mathrm{Gal}(F^\mathrm{al}/F)$ fixes $L$ if and only if
$(E, x, \Theta)^\sigma$ is 
isomorphic (over $F^\mathrm{al}$) to $(E,x,\Theta)$.
\end{Def}

\begin{Rem}\label{bold notation}
We will often use $(E,\mathbf{C},\Theta)$ to 
denote a $\Lambda$-augmented $\Gamma_0(\mathbf{N})$
structure, and write
$C\subset\mathbf{C}$ for the $\Gamma_0(N)$ structure obtained by
forgetting the $\Gamma_0(M)$ structure.
\end{Rem}

If $z\in Y_1(N,M)_{/L}$ is a closed point we let 
$E^\mathrm{univ}_z$ be the pullback of the universal elliptic curve 
over $Y_1(N,M)_{/L}$ to $k(z)$.
Define the module of \emph{$\Lambda$-augmented cycles on $Y_1(N,M)_{/L}$}
\begin{displaymath}
\mathcal{A}_\Lambda^\circ(Y_1(N,M)_{/L})=
\bigoplus_z \mathcal{A}_\Lambda^\circ(E_z^\mathrm{univ}),
\end{displaymath}
where the sum is over all closed points of $Y_1(N,M)_{/L}$
(note that $\mathcal{A}_\Lambda^\circ(E_z^\mathrm{univ})$ means the points of
$\mathcal{A}_\Lambda(E_z^\mathrm{univ})$ defined over the field of definition 
of $E_z^\mathrm{univ}$, $k(z)$, \emph{not} over $L$).
For any set of closed points $Z\subset Y_1(N,M)_{/L}$ define
\begin{displaymath}
\mathcal{A}_\Lambda^\circ(Z;Y_1(N,M)_{/L})
\end{displaymath} in the same way, but with 
the sum restricted to $z\in Z$.  We also define
\begin{equation}
\label{aug sum}
\mathcal{A}_\Lambda(\Gamma_1(N,M))=\bigoplus_{(E,x)}\mathcal{A}_\Lambda(E),
\end{equation}
where the sum is over isomorphism classes of $\Gamma_1(N,M)$ structures
over $F^\mathrm{al}$. A $\Lambda$-augmented 
$\Gamma_0(N)$-structure $(E,x,\Theta)$ over $F^\mathrm{al}$ defines an element
of the modular (\ref{aug sum}), denoted the same way, by taking
the element $\Theta$ in the summand attached to $(E,x)$ and $0$
in the other summands. The module $\mathcal{A}_\Lambda(\Gamma_1(N,M))$ 
has a natural action of $\mathrm{Gal}(F^\mathrm{al}/L)$, and 
\begin{equation}\label{moduli cycles}
\mathcal{A}_\Lambda^\circ(Y_1(N,M)_{/L})\cong 
\mathcal{A}_\Lambda(\Gamma_1(N,M))^{\mathrm{Gal}(F^\mathrm{al}/L)}.
\end{equation}
Indeed, a   closed point $z\in Y_1(N,M)_{/L}$ and a
$\Theta\in\mathcal{A}_\Lambda^\circ(E_z^\mathrm{univ})$ determine 
a $\Lambda$-augmented $\Gamma_1(N,M)$ structure
$(E^\mathrm{univ}_z,x^\mathrm{univ}_z,\Theta)$ over $k(z)$, 
where $(E_z^\mathrm{univ},x_z^\mathrm{univ})$ is the pullback to 
$k(z)$ of the universal $\Gamma_1(N,M)$ structure.
Each embedding of $L$-algebras $k(z)\hookrightarrow F^\mathrm{al}$ 
then determines 
a $\Lambda$-augmented $\Gamma_1(N,M)$ structure
over $F^\mathrm{al}$, and summing over all embeddings 
$k(z)\hookrightarrow F^\mathrm{al}$
determines an element of the right hand side of (\ref{moduli cycles}).
Extending linearly over all $z$ and $\Theta$ gives the desired map.  
The construction of the inverse is similar and easy.

%%%%%%%%%%%%%%%%%%%%%%%%%%%%%%%%%%%%%%%%%%%%%%%%%%%%%%%%%%%%%

\subsection{A higher weight Kummer map}
\label{Gamma_1}

%%%%%%%%%%%%%%%%%%%%%%%%%%%%%%%%%%%%%%%%%%%%%%%%%%%%%%%%%%%%

In this subsection $M=1$.
Let $L\subset F^\mathrm{al}$ be an algebraic extension of $F$.
Fix a closed point $z\in Y_1(N)_{/L}$ and
and write $i_z$ for the closed immersion 
$\mathrm{Spec}(k(z))\map{}Y_1(N)_{/L}$.
Denote by $j:Y_1(N)_{/F^\mathrm{al}}\hookrightarrow 
X_1(N)_{/F^\mathrm{al}}$ the usual 
compactification.

\begin{Lem}\label{purity}
There are canonical isomorphisms
\begin{displaymath}
\mathcal{A}_\Lambda^\circ(E_z^\mathrm{univ})
\cong
H^0(z,i_z^*\mathcal{L}_\Lambda(k-1))
\cong
H_{z}^2(Y_1(N)_{/L},\mathcal{L}_{\Lambda}(k)).
\end{displaymath}
\end{Lem}

\begin{proof}
The first isomorphism is induced by the isomorphism (\ref{nice sheaf iso}),
and the second is a consequence of cohomological purity as in
\cite[Chapter VI \S5]{milne:etale}.
%fix a geometric point $\bar{z}$ above $z$ and take 
%$\mathrm{Aut}(\bar{z}/z)$-invariants of the stalk at $\bar{z}$ of
%\begin{displaymath}
%\big(\mathrm{Sym}^{2k-2}(\underline{E}^\mathrm{univ}[m])\big)(1-k)
%\cong\mathcal{L}_\Lambda(k-1).
%\end{displaymath}
%The second isomorphism in the statement of the lemma is the Gysin
%map: cohomological purity \cite[Chapter VI \S5]{milne:etale}
%gives an isomorphism
%\begin{displaymath}
%H^0\big(z, i_z^*\mathcal{L}(k-1)\otimes 
%\underline{H}^2_z(Y_1(N)_{/L},\underline{\mu}_{m})\big) 
%\cong H^2_z(Y_1(N)_{/L},\mathcal{L}_\Lambda(k))
%\end{displaymath}
%for a sheaf $\underline{H}^2_z(Y_1(N)_{/L},\underline{\mu}_{m})$ 
%on $z$ which is locally 
%isomorphic to $\underline{\Lambda}$, and whose global sections
%are identified with $H^2_z(Y_1(N)_{/L},\underline{\mu}_{m})$.  
%The composition
%\begin{displaymath}
%\Lambda\cong H^1_z(Y_1(N)_{/L},\mathbb{G}_m)
%\otimes_\mathbb{Z}\Lambda \hookrightarrow
%H^2_z(Y_1(N)_{/L},\underline{\mu}_{m})
%\end{displaymath}
%(the isomorphism by \cite[Chapter VI \S 6]{milne:duality}, the injection by
%the Kummer sequence) exhibits a nontrivial global section of order $m$, so 
%$\underline{H}^2_z(Y_1(N)_{/L},\underline{\mu}_{m})\cong \Lambda$.
\end{proof}

\begin{Lem}\label{HS}
There is a canonical isomorphism
$$
H^2(X_1(N)_{/L},j_*\mathcal{L}_\Lambda(k))
\cong
H^1(F^\mathrm{al}/L, 
\tilde{H}^1(Y_1(N)_{/F^\mathrm{al}},\mathcal{L}_\Lambda(k)).
$$
%There are canonical isomorphisms
%\begin{displaymath}
%H^2(Y_1(N)_{/L}, \mathcal{L}_{\Lambda}(k))  \cong
%H^1(F^\mathrm{al}/L, H^1(Y_1(N)_{/F^\mathrm{al}},\mathcal{L}_\Lambda)(k))
%\end{displaymath}
%\begin{displaymath}
%H^2_c(Y_1(N)_{/L}, \mathcal{L}_{\Lambda}(k))  \cong
%H^1(F^\mathrm{al}/L, H^1_c(Y_1(N)_{/F^\mathrm{al}},\mathcal{L}_\Lambda)(k)).
%\end{displaymath}
\end{Lem}

\begin{proof}
One checks directly that $\mathrm{Sym}^{2k-2}\Lambda^2$ has 
no $\mathrm{SL}_2(\Lambda)$-invariants, and hence, by the discussion 
of \S \ref{sheaves}, $H^0(Y_1(N)_{/F^\mathrm{al}},\mathcal{L}_\Lambda)=0$.
Using Poincar\'e duality we see also that the group
$$
H^2(X_1(N)_{/F^\mathrm{al}},j_*\mathcal{L})\cong
H^2_c(Y_1(N)_{/F^\mathrm{al}},\mathcal{L})
$$
is trivial, and so
\begin{equation}\label{degree one}
H^i(X_1(N)_{/F^\mathrm{al}}, j_*\mathcal{L}_\Lambda)=0
\end{equation}
for $i\not=1$.  Thus the Hochschild-Serre spectral sequence
and the identification
\begin{equation}\label{well known}
H^1(X_1(N)_{/F^\mathrm{al}},j_*\mathcal{L}_\Lambda)\cong
\tilde{H}^1(Y_1(N)_{/F^\mathrm{al}},\mathcal{L}_\Lambda)
\end{equation}
yield the desired isomorphism.
\end{proof}

\begin{Def}\label{kummer}
Combining Lemmas \ref{purity} and \ref{HS} with the homomorphism
$$
H_{z}^2(Y_1(N)_{/L},\mathcal{L}_{\Lambda}(k))
\map{}
H^2(X_1(N)_{/L},j_*\mathcal{L}_\Lambda(k))
$$ 
we obtain a map
\begin{displaymath}
\mathcal{A}_\Lambda^\circ(E_z^\mathrm{univ})
\map{}
H^1\big(F^\mathrm{al}/L,\tilde{H}^1(Y_1(N)_{/F^\mathrm{al}},
\mathcal{L}_\Lambda)(k)\big)
\end{displaymath}
for each closed point $z\in Y_1(N)_{/L}$.  This map
extends linearly to define the \emph{$\Lambda$-augmented Kummer map}
\begin{displaymath}
\mathcal{A}_\Lambda^\circ(Y_1(N)_{/L})\map{}H^1\big(F^\mathrm{al}/L, 
\tilde{H}^1(Y_1(N)_{/F^\mathrm{al}},\mathcal{L}_\Lambda)(k)\big).
\end{displaymath}
\end{Def}

We now give an alternate definition of the $\Lambda$-augmented
Kummer map.  The proof of the equivalence of the two definitions 
requires only minor modification of 
\cite[Lemma 9.4]{Jannsen} and is omitted.
  Given a closed point $z\in Y_1(N)_{/L}$, 
let $U=U_{/F^\mathrm{al}}$ be the open complement of 
$z\times_L F^\mathrm{al}$ in $X_1(N)_{/F^\mathrm{al}}$.
Excision and the relative cohomology sequence
give the exact sequence
\begin{equation}\label{relative}
0\map{} H^1(X_1(N)_{/F^\mathrm{al}},j_*\mathcal{L}_\Lambda) 
\map{}
H^1(U,j_*\mathcal{L}_\Lambda)
\map{} 
H^2_{z\times_L F^\mathrm{al}}(Y_1(N)_{/F^\mathrm{al}},\mathcal{L}_\Lambda)
\map{}0
\end{equation}
where the initial and terminating zeros are  justified by 
cohomological purity and (\ref{degree one}), respectively.
Using Lemma \ref{purity} we may identify
$\mathcal{A}_\Lambda^\circ(E_z^\mathrm{univ})$
with the $\mathrm{Gal}(F^\mathrm{al}/L)$-invariants of
\begin{displaymath}
\bigoplus_{w\in z\times_L F^\mathrm{al}}\mathcal{A}_\Lambda
(E_w^\mathrm{univ})\cong
H^2_{z\times_L F^\mathrm{al}}(Y_1(N)_{/F^\mathrm{al}},\mathcal{L}_\Lambda)(k),
\end{displaymath}
and the connecting homomorphism
\begin{equation}\label{alt kummer}
\mathcal{A}_\Lambda^\circ(E_z^\mathrm{univ}) \map{}
H^1(F^\mathrm{al}/L, H^1(X_1(N)_{/F^\mathrm{al}},j_*\mathcal{L}_\Lambda)(k) )
\end{equation}
then agrees with Definition \ref{kummer}, using the identification
of (\ref{well known}).

For $L/F$ any algebraic extension, the group $\Delta$ acts  
on $Y_1(N,M)_{/L}$ through the diamond automorphisms.
There is a similar action of  $\Delta$ on $\mathcal{A}_\Lambda(\Gamma_1(N,M))$
commuting with the $\mathrm{Gal}(F^\mathrm{al}/L)$-action, 
and so $\Delta$ also acts
on $\mathcal{A}_\Lambda^\circ(Y_1(N,M)_{/L})$ by (\ref{moduli cycles}),
and on $\mathcal{A}_\Lambda^\circ(Z;Y_1(N,M)_{/L})$ for any subset
$Z\subset Y_1(N,M)_{/L}$ stable under $\Delta$.
There is also a familiar action of $\Delta$ on
the cohomology $H^i(Y_1(N)_{/L},\mathcal{L}_\Lambda(j))$ for any $i$ and $j$,
on compactly supported cohomology,
and on the cohomology supported on $Z$ for any closed set
$Z\subset Y_1(N)_{/L}$ stable under $\Delta$.
The action of $\Delta$ is compatible with
the $\Lambda$-augmented Kummer map of Definition 
\ref{kummer}.

%%%%%%%%%%%%%%%%%%%%%%%%%%%%%%%%%%%%%%%%%%%%%%%%%%%%%%%%%

\subsection{Augmented $\Gamma_0(N)$ structures}
\label{Gamma_0}

%%%%%%%%%%%%%%%%%%%%%%%%%%%%%%%%%%%%%%%%%%%%%%%%%%%%%%%%%%

Now fix a $\Lambda$-augmented $\Gamma_0(\mathbf{N})$ structure 
$(E,\mathbf{C},\Theta)$ over $F^\mathrm{al}$ and suppose 
$L$ is a finite extension 
of $F$ containing the field of moduli of   $(E,\mathbf{C},\Theta)$.
In particular $L$ contains
the field of moduli (in the usual sense) of the pair $(E,\mathbf{C})$,
and so determines a closed point $y\in Y_0(\mathbf{N})_{/L}$ 
with residue field $L$. Let $Z\subset Y_1(N,M)_{/L}$ 
denote the set of closed points lying above $y$
under the forgetful degeneracy map 
\begin{displaymath}
F_{N,M}:Y_1(N,M)_{/L}\map{}Y_0(\mathbf{N})_{/L}.
\end{displaymath}
Let $P_1,\ldots, P_{\varphi(N)}$  be the generators of $C$ 
(using the convention of Remark 
\ref{bold notation}) and let $x_i$ be the $\Gamma_1(N,M)$
structure on $E$ determined by $P_i$ and the $\Gamma_0(M)$ structure
underlying $\mathbf{C}$.  Define, using (\ref{moduli cycles}),
\begin{equation}\label{the class}
F^*_{N,M}(E,\mathbf{C},\Theta)
=\sum_{1=1}^{\varphi(N)} (E,x_i,\Theta)
\in \mathcal{A}_\Lambda^\circ(Z;Y_1(N,M)_{/L})^\Delta.
\end{equation}

Taking $M=1$ for the moment, we denote by 
\begin{equation}\label{classes}
\Omega_{L}(E,C,\Theta)\in 
H^1\big(F^\mathrm{al}/L,
\tilde{H}^1(Y_1(N)_{/F^\mathrm{al}},\mathcal{L}_\Lambda)(k)\big)^\Delta
\end{equation}
the image of $F^*_N(E,C,\Theta)$ under the $\Lambda$-augmented
Kummer map of Definition \ref{kummer}.  Allowing $L$ to vary
over all finite extensions of $F$ containing the field of
moduli of $(E,C,\Theta)$, the formation of
$\Omega_L(E,C,\Theta)$ is compatible with the restriction
maps on Galois cohomology.

%%%%%%%%%%%%%%%%%%%%%%%%%%%%%%%%%%%%%%%%%%%%%%%%%%%%%%%%%%

\subsection{Reduction and ramification}

%%%%%%%%%%%%%%%%%%%%%%%%%%%%%%%%%%%%%%%%%%%%%%%%%%%%%%%%%%

In this subsection we assume that $M=1$ and $F$ is a finite
extension of $\mathbb{Q}_q$ for some prime $q\nmid \ell N$.
Let $\mathbb{F}^\mathrm{al}$ and $\mathbb{F}$ denote the residue fields of 
$F^\mathrm{al}$ and $F$, respectively, so that
\begin{equation}\label{deligne base change}
W_\Lambda\stackrel{\mathrm{def}}{=}
\tilde{H}^1(Y_1(N)_{/F^\mathrm{al}},\mathcal{L}_\Lambda)
\cong \tilde{H}^1(Y_1(N)_{/\mathbb{F}^\mathrm{al}},\mathcal{L}_\Lambda)
\end{equation}
is an unramified $\mathrm{Gal}(F^\mathrm{al}/F)$-module.
%\cite[Theorem 7.4.1.2]{conrad}.
Let $(E,C,\Theta)$ be a $\Lambda$-augmented $\Gamma_0(N)$ structure
over $F^\mathrm{al}$, and assume that $E$ has good reduction.
The reduction of $(E,C)$, a $\Gamma_0(N)$-structure
over the field $\mathbb{F}^\mathrm{al}$, is denoted
$(\mathrm{red}(E),\mathrm{red}(C))$, and we identify 
$E(F^\mathrm{al})[m]$ with 
$\mathrm{red}(E)(\mathbb{F}^\mathrm{al})[m]$ as $\Lambda$-modules.
This determines an isomorphism
$
\mathrm{red}:\mathcal{A}_\Lambda(E)\cong \mathcal{A}_\Lambda(\mathrm{red}(E)),
$ 
and so we obtain a $\Lambda$-augmented $\Gamma_0(N)$ structure
\begin{displaymath}
\mathrm{red}(E,C,\Theta)=(\mathrm{red}(E),\mathrm{red}(C),\mathrm{red}(\Theta))
\end{displaymath}
over $\mathbb{F}^\mathrm{al}$. If 
$L\subset F^\mathrm{al}$ is a finite extension of $F$ 
containing the field of moduli
of $(E,C,\Theta)$ then the residue field of $L$, $\mathbb{L}$,
contains the field of moduli of $\mathrm{red}(E,C,\Theta)$,
and so we may form 
\begin{displaymath}
\Omega_{\mathbb{L}}(\mathrm{red}(E),\mathrm{red}(C),\mathrm{red}(\Theta))\in 
H^1(\mathbb{F}^\mathrm{al}/\mathbb{L},W_\Lambda(k))^\Delta.
\end{displaymath}

\begin{Prop}\label{Prop:reduction}
Suppose $\ell\nmid \varphi(N)$.
In the notation above, $\Omega_{L}(E,C,\Theta)$
is equal to the image of 
$\Omega_{\mathbb{L}}(\mathrm{red}(E),\mathrm{red}(C),\mathrm{red}(\Theta))$
under the inflation map
\begin{displaymath}
H^1(\mathbb{F}^\mathrm{al}/\mathbb{L}, 
W_\Lambda(k))\cong H^1(L^\mathrm{unr}/L,W_\Lambda(k))
\map{}H^1(F^\mathrm{al}/L,W_\Lambda(k)),
\end{displaymath}
where $L^\mathrm{unr}\subset F^\mathrm{al}$ is the 
maximal unramified extension of $L$.  
\end{Prop}

\begin{proof}
It is clear from the definition that the construction
\begin{displaymath}
(E,C,\Theta)\mapsto F^*_{N}(E,C,\Theta)
\end{displaymath}
is compatible with reduction, so the proof of the proposition amounts
to verifying that the constructions of \S \ref{Gamma_1} extend
across integral models.  Let $Z\subset Y_1(N)_{/L}$ be as in \S
\ref{Gamma_0} and suppose for the moment that every $z\in Z$ has
residue field $L$.  Denoting the integer ring of $L$ by $\mathcal{O}_L$,
each $z\in Y_1(N)_{/L}$ extends to a smooth section
$\underline{z}:\mathrm{Spec}(\mathcal{O}_L)\map{}X_1(N)_{/\mathcal{O}_L}$ 
of the canonical integral model 
of $X_1(N)$ over $\mathcal{O}_L$.  Since $E_z^\mathrm{univ}$ 
has potentially good reduction,
this section does not meet the cusps in the special fiber, and so
factors through the affine subscheme $Y_1(N)_{/\mathcal{O}_L}$.  
The sequence (\ref{relative}) extends across 
integral models over the integer ring of the maximal unramified 
extension of $L$, denoted $R$, to give the middle row of the commutative
diagram
\begin{displaymath}
\xymatrix{
0\ar[r]& 
H^1(X_{/\mathbb{F}^\mathrm{al}},j_*\mathcal{L}_\Lambda)\ar[r]&
H^1(U_{/\mathbb{F}^\mathrm{al}},j_*\mathcal{L}_\Lambda)\ar[r]&
H^2_{\underline{z}\times\mathbb{F}^\mathrm{al}}
(Y_{/\mathbb{F}^\mathrm{al}},\mathcal{L}_\Lambda)\ar[r]&
0
\\
0\ar[r]& 
H^1(X_{/R},j_*\mathcal{L}_\Lambda)\ar[r]\ar[d]\ar[u]&
H^1(U_{/R},j_*\mathcal{L}_\Lambda)\ar[r]\ar[d]\ar[u]&
H^2_{\underline{z}\times R}(Y_{/R},\mathcal{L}_\Lambda)\ar[r]\ar[d]\ar[u]&
0\\
0\ar[r]& 
H^1(X_{/F^\mathrm{al}},j_*\mathcal{L}_\Lambda)\ar[r]&
H^1(U_{/F^\mathrm{al}},j_*\mathcal{L}_\Lambda)\ar[r]&
H^2_{\underline{z}\times F^\mathrm{al}}
(Y_{/F^\mathrm{al}},\mathcal{L}_\Lambda)\ar[r]&
0
}
\end{displaymath}
where we abbreviate $X=X_1(N)$ and $Y=Y_1(N)$, and write
$U_{/R}$ for the complement of $\underline{z}\times R$ in $X_{/R}$.
Lemma \ref{purity} implies that the rightmost vertical arrows are
isomorphisms (the pullback of  $\mathcal{L}_\Lambda$ 
to $\underline{z}\times R$ is constant, so its global sections can be computed 
in either geometric fiber).  The vertical arrows on the left are isomorphisms,
by (\ref{well known})  and the isomorphism of 
(\ref{deligne base change}).  By the five lemma
the arrows in the middle column are isomorphisms as well, and
from this it follows that the map (\ref{alt kummer}) is compatible with
reduction.

For the general case, choose a $z\in Z$ and an embedding 
$k(z)\hookrightarrow F^\mathrm{al}$. Let $L'$ be the image of this embedding.
It is easily seen that $L'$ does not depend on the point $z$ or the choice of
embedding, that $L'/L$ is a Galois extension of degree dividing
$\varphi(N)$, and that every point in 
$Z\times_L L'\hookrightarrow Y_1(N)_{/L'}$ 
has residue field $L'$.  The proposition follows from the bijectivity of the
restriction map
\begin{displaymath}
H^1(F^\mathrm{al}/L, W_\Lambda(k))\cong 
H^1(F^\mathrm{al}/L', W_\Lambda(k))^{\mathrm{Gal}(L'/L)},
\end{displaymath}
and of the analogous map on the level of residue fields, 
together with the special case considered above.
\end{proof}

%%%%%%%%%%%%%%%%%%%%%%%%%%%%%%%%%%%%%%%%%%%%%%%%%%%%%%%%%%%%

\subsection{Degeneracy maps}
\label{degen}

%%%%%%%%%%%%%%%%%%%%%%%%%%%%%%%%%%%%%%%%%%%%%%%%%%%%%%%%%%%%

Recall that $M$ is squarefree.
Given a divisor $M'\mid M$ we define a degeneracy map
\begin{displaymath}
\alpha^M_{M'}: \mathcal{A}_\Lambda(\Gamma_1(N,M))\map{}
\mathcal{A}_\Lambda(\Gamma_1(N,M'))
\end{displaymath}
as follows. Given a $\Lambda$-augmented $\Gamma_1(N,M)$
structure $(E,x,\Theta)$ over $F^\mathrm{al}$, we define
\begin{displaymath}
\alpha^M_{M'}(E,x,\Theta)=(E,x',\Theta)
\end{displaymath}
where $x'$ is the $\Gamma_1(N,M')$ structure on $E$
underlying $x$.  Extend this $\Lambda$-linearly to a map
on $\mathcal{A}_\Lambda(\Gamma_1(N,M))$. We also define
a degeneracy map
\begin{displaymath}
\beta^M_{M'}: \mathcal{A}_\Lambda(\Gamma_1(N,M))\map{}
\mathcal{A}_\Lambda(\Gamma_1(N,M'))
\end{displaymath}
as follows.  Given a $\Lambda$-augmented $\Gamma_1(N,M)$
structure $(E,x,\Theta)$ over $F^\mathrm{al}$ let $P$ and $D$ be
the $\Gamma_1(N)$ and $\Gamma_0(M)$ structures underlying $x$.
Let $D_0\subset D$ be the subgroup of order $M/M'$, let
$E'=E/D_0$, and let $P'$ and $D'$ be the images of
$P$ and $D$ under $E\map{}E'$.
Let $\Theta'$ be the image of $\Theta$
under $\mathcal{A}_\Lambda(E)\map{}\mathcal{A}_\Lambda(E')$.
Write $x'$ for the $\Gamma_1(N,M')$ structure $(P',D')$ on $E'$.
Now define
\begin{displaymath}
\beta^M_{M'}(E,x,\Theta)=(E',x',\Theta')
\end{displaymath}
and again extend linearly. The maps $\alpha^M_{M'}$ and  $\beta^M_{M'}$
respect the $\mathrm{Gal}(F^\mathrm{al}/F)$ action, and so induce maps
\begin{displaymath}
\alpha^M_{M'},\beta^M_{M'}:\mathcal{A}_\Lambda^\circ(Y_1(N,M)_{/L})\map{}
\mathcal{A}_\Lambda^\circ(Y_1(N,M')_{/L})
\end{displaymath}
for any algebraic extension $L/F$.

%%%%%%%%%%%%%%%%%%%%%%%%%%%%%%%%%%%%%%%%%%%%%%%%%%%%%%%%%%%%%%%
%%%%%%%%%%%%%%%%%%%%%%%%%%%%%%%%%%%%%%%%%%%%%%%%%%%%%%%%%%%%%%

\section{Families of augmented CM points}
\label{S: families}

%%%%%%%%%%%%%%%%%%%%%%%%%%%%%%%%%%%%%%%%%%%%%%%%%%%%%%%%%%%%%%
%%%%%%%%%%%%%%%%%%%%%%%%%%%%%%%%%%%%%%%%%%%%%%%%%%%%%%%%%%%%%%%%%

Let $\Lambda=\mathbb{Z}/m\mathbb{Z}$ for an $\ell$-power $m$.
Fix a quadratic imaginary field $K\subset\mathbb{Q}^\mathrm{al}$,
assume that all prime divisors of $N$ are split in $K$,
and fix an ideal $\mathfrak{N}\subset\mathcal{O}_K$ such that 
$\mathcal{O}_K/\mathfrak{N}\cong\mathbb{Z}/N\mathbb{Z}$.
Fix an elliptic curve $E_1$ over 
$\mathbb{Q}^\mathrm{al}$ with complex multiplication by the
maximal order $\mathcal{O}_K$ 
(there are $\#\mathrm{Pic}(\mathcal{O}_K)$ such curves).  
Let $j:\mathcal{O}_K\hookrightarrow
\mathrm{End}_{\mathbb{Q}^\mathrm{al}}(E_1)$ be
normalized so that pullback by $j(\alpha)$ acts as multiplication by
$\alpha$ on the cotangent space of $E_1(\mathbb{C})$ for every $\alpha\in K$.
Set $C_1=E_1[\mathfrak{N}]$, a cyclic subgroup of order $N$.

\begin{Def}\label{cyclic}
  An element $c\in\mathrm{GL}_2(\mathbb{Q}_p)$ is 
\emph{cyclic} if $(c^{-1}\mathbb{Z}_p^2)$
  contains $\mathbb{Z}_p^2$ with cyclic quotient.  The \emph{degree} of a
  cyclic $c$ is
  \begin{displaymath}
  \deg(c)=[c^{-1}\mathbb{Z}_p^2:\mathbb{Z}_p^2]=p^{\mathrm{ord}_p(\det(c))}.
  \end{displaymath}
\end{Def}

%%%%%%%%%%%%%%%%%%%%%%%%%%%%%%%%%%%%%%%%%%%%%%%%%%%%%%%%%%%%%%%

\subsection{A parametrized family of Heegner points}
\label{tree}

%%%%%%%%%%%%%%%%%%%%%%%%%%%%%%%%%%%%%%%%%%%%%%%%%%%%%%%%%%%%%%

A choice of isomorphism of $\mathbb{Z}_p$-modules 
$\mathrm{Ta}_p(E_1)\cong\mathbb{Z}_p^2$ (which
we now fix) determines a family of elliptic curves
over $\mathbb{Q}^\mathrm{al}$ parametrized by 
\begin{displaymath}
\mathcal{T}=\mathbb{Q}_p^\times\mathrm{GL}_2(\mathbb{Z}_p) 
\backslash \mathrm{GL}_2(\mathbb{Q}_p)
\end{displaymath}
as follows.  For each cyclic subgroup $X\subset E_1[p^\infty]\cong
(\mathbb{Q}_p/\mathbb{Z}_p)^2$ there is a cyclic 
$c_X\in \mathrm{GL}_2(\mathbb{Q}_p)$ such that
$X=(c_X^{-1}\mathbb{Z}_p^2)/\mathbb{Z}_p^2$.  
The assignment $X\mapsto c_X$ establishes
a bijection between the set of such subgroups and the cyclic elements
of $\mathrm{GL}_2(\mathbb{Q}_p)$, modulo left 
multiplication by $\mathrm{GL}_2(\mathbb{Z}_p)$.  We
denote the inverse by $c\mapsto X_c$. The projection map
$\mathrm{GL}_2(\mathbb{Z}_p)\backslash
\mathrm{GL}_2(\mathbb{Q}_p)\map{}\mathcal{T}$ establishes a bijection
between the left $\mathrm{GL}_2(\mathbb{Z}_p)$-orbits 
of cyclic elements and the set
$\mathcal{T}$. To each $g\in\mathcal{T}$ we then define
$X_g\subset E_1[p^\infty]$ to be $X_c$ for any cyclic 
$c\in\mathrm{GL}_2(\mathbb{Q}_p)$
lifting $g$, and define a cyclic
$p$-power isogeny
\begin{displaymath}
f_g: E_1\map{}E_g=E_1/X_{g}
\end{displaymath}
of degree $\deg(c)$.  Define the \emph{degree} of $g$ by
$\deg(g)=\deg(c)=\deg(f_g)$.
The elliptic curve $E_g$ over $\mathbb{Q}^\mathrm{al}$ 
inherits a $\Gamma_0(N)$ structure 
$C_g=f_g(C_1)=E_g[\mathfrak{N}\cap\mathcal{O}_g]$,
where $\mathcal{O}_g\subset\mathcal{O}_K$ is the largest order
which leaves the subgroup $X_g\subset E_1[p^\infty]$ stable.
The conductor of $\mathcal{O}_g$ is a power of $p$.

For each $g\in\mathcal{T}$ let
$H_g$ be the ring class field of $\mathcal{O}_g$, thus $H_g/K$ is Galois
with Galois group canonically identified with $\mathrm{Pic}(\mathcal{O}_g)$.
The Weil pairing on $E_1[m]$ provides a canonical (up to $\pm 1$)
isomorphism between $(\mathrm{Sym}^2 E_1[m])(-1)$ and the traceless
$\Lambda$-module endomorphisms of $E_1[m]$.  In particular there is a
canonical (up to sign) element $\vartheta_1\in (\mathrm{Sym}^2 E_1[m])(-1)$
corresponding to $\sqrt{D}\in \mathrm{End}_{\mathbb{Q}^\mathrm{al}}(E_1)$.  
Let $\Theta_1 \in\mathcal{A}_{\Lambda}(E_1)$ be the image of
$\vartheta_1^{k-1}$ under the natural projection
\begin{displaymath}
\mathrm{Sym}^{k-1}\big((\mathrm{Sym}^2 E_1[m])(-1)\big)\map{}
\big(\mathrm{Sym}^{2k-2}E_1[m]\big)(1-k),
\end{displaymath}
and define 
$\Theta_g=f_g(\Theta_1) \in \mathcal{A}_\Lambda(E_g).$
The data $K$, $E_1$, $\mathfrak{N}$, and 
$\mathrm{Ta}_p(E_1)\cong\mathbb{Z}_p^2$ thus determine a family
$g\mapsto (E_g,C_g,\Theta_g)$
of $\Lambda$-augmented $\Gamma_0(N)$ structures over 
$\mathbb{Q}^\mathrm{al}$ parametrized 
by $\mathcal{T}$.  
\emph{This data is to remain fixed throughout the remainder 
of the article.}
%\begin{Prop}
%For every $g\in\mathcal{T}$, the field of moduli 
%(Definition \ref{moduli definition}, with $F=K$)
%of $(E_g,C_g,\Theta_g)$  is $H_g$.
%\end{Prop}
%\begin{proof}
%Let $j$ be the usual modular $j$-invariant.
%By the theory of complex multiplication $K(j(E_1))=H$, the Hilbert
%class field of $K$, and so we 
%may fix a model, $A$, of $E_1$ over $H$ and an isomorphism
%$E_1\cong A\times_H \mathbb{Q}^\mathrm{al}$.  Define $\vartheta_A$
%and $\Theta_A$ exactly as above, with $E_1$ replaced by $A$.
%As all endomorphisms of
%$A$ are defined over $H$, the subgroup $C_A=A[\mathfrak{N}]$
%is defined over $H$, and $\vartheta_A$ is fixed by 
%$\mathrm{Gal}(\mathbb{Q}^\mathrm{al}/H)$.
%It follows easily that $\Theta_A\in\mathcal{A}^\circ_\Lambda(A)$,
%and so the $\Lambda$-augmented $\Gamma_0(N)$ structure 
%$(E_1,C_1,\Theta_1)$ has a model $(A,C_A,\Theta_A)$ defined over $H$. 
%Again by the theory of
%complex multiplication, the subgroup of $A$ corresponding to $X_g$ 
%is defined over $H_g$, and it follows easily that $(E_g,C_g,\Theta_g)$
%has a model over $H_g$ for every $g\in\mathcal{T}$.  
%Hence the field of moduli of 
%$(E_g,C_g,\Theta_g)$ is contained in $H_g$, 
%and so must be equal to $H_g$ as $K(j(E_g))=H_g$.
%\end{proof}
Define a $\mathrm{Gal}(\mathbb{Q}^\mathrm{al}/K)$-module
\begin{equation}\label{our module}
W_\Lambda=\tilde{H}^1(Y_1(N)_{/\mathbb{Q}^\mathrm{al}},\mathcal{L}_\Lambda).
\end{equation}
Using the theory of complex multiplication it is easily seen that
the field of moduli of $(E_g,C_g,\Theta_g)$ is $H_g$, and so
the construction (\ref{classes}) yields a family of cohomology classes
parametrized by $g\in\mathcal{T}$ 
\begin{equation}
\label{Heegner family}
g\mapsto \Omega_{H_g}(E_g,C_g,\Theta_g)
\in H^1(\mathbb{Q}^\mathrm{al}/H_g,W_\Lambda(k))^\Delta.
\end{equation}
Let $H[p^s]$ denote the ring class field of conductor $p^s$ of $K$.
For each $s\ge 0$ define $\mathcal{T}_s\subset \mathcal{T}$ to be the 
subset consisting of all $g$ such that $H_g\subset H[p^s]$. 
For $g\in\mathcal{T}_s$ we let 
\begin{equation}\label{class param}
\Omega_s(g) = \Omega_{H[p^s]}(E_g,C_g,\Theta_g)
\in H^1(\mathbb{Q}^\mathrm{al}/H[p^s],W_\Lambda(k))^\Delta
\end{equation}
be the restriction of the cohomology class of (\ref{Heegner family})
to $\mathrm{Gal}(\mathbb{Q}^\mathrm{al}/H[p^s])$.

%%%%%%%%%%%%%%%%%%%%%%%%%%%%%%%%%%%%%%%%%%%%%%%%%%%%%%%%%%%%%%%%%%%

\subsection{Level $M$ structure}
\label{level M}

%%%%%%%%%%%%%%%%%%%%%%%%%%%%%%%%%%%%%%%%%%%%%%%%%%%%%%%%%%%%%%%%%%%%

Suppose that the integer $M$ of \S \ref{notation} is a divisor of
$\mathrm{disc}(K)$ and let $\mathfrak{M}$ be the unique $\mathcal{O}_K$-ideal
of norm $M$.  Although we assumed in \S \ref{tree} that 
all prime divisors of $N$ are split in $K$, the constructions of
\S \ref{tree} 
work equally well with $N$ replaced by $\mathbf{N}=NM$ and $\mathfrak{N}$
replaced by $\mathfrak{N M}$.  This has the effect of endowing
each $(E_g,C_g)$ with the extra $\Gamma_0(\mathbf{N})$ 
structure $\mathbf{C}_g=E_g[\mathfrak{NM}\cap \mathcal{O}_g]$, so that
$g\mapsto (E_g, \mathbf{C}_g, \Theta_g)$ is a parametrized family of
$\Lambda$-augmented $\Gamma_0(\mathbf{N})$ structures.

%%%%%%%%%%%%%%%%%%%%%%%%%%%%%%%%%%%%%%%%%%%%%%%%%%%%%%%%%
%%%%%%%%%%%%%%%%%%%%%%%%%%%%%%%%%%%%%%%%%%%%%%%%%%%%%%%
 
\section{Reduction of the family}
\label{S: reduction of the family}

%%%%%%%%%%%%%%%%%%%%%%%%%%%%%%%%%%%%%%%%%%%%%%%%%%%%%%%%%%
%%%%%%%%%%%%%%%%%%%%%%%%%%%%%%%%%%%%%%%%%%%%%%%%%%%%%%%

In this section we prove Theorem \ref{chaos}, which is our analogue of 
\cite[Theorem 3.1]{cornut02}.
Suppose $\ell\not=p$ and take $M$ to be a 
divisor of $\mathrm{disc}(K)$ as in \S \ref{level M}.  
Assume $\ell\nmid \varphi(N)$ and $\ell>2k-2$. 
Let $\Lambda$ be a finite quotient of $\mathbb{Z}_\ell$.  
Let $\mathfrak{Q}$ be a finite set of rational primes inert in $K$,
all prime to $\ell p \mathbf{N}$.  For each $q\in\mathfrak{Q}$, let
$\mathfrak{q}$ denote the prime of $K$ above $q$ and fix an extension
of $\mathfrak{q}$ to a place of $\mathbb{Q}^\mathrm{al}$.  
We will abusively use 
$\mathfrak{Q}$ to refer to the set of rational primes, 
the set of primes of $K$ above them,
and also the set of chosen places of $\mathbb{Q}^\mathrm{al}$.
Let $\mathbb{F}_\mathfrak{q}^\mathrm{al}$ and $\mathbb{F}_\mathfrak{q}$ 
denote the residue fields of
$\mathbb{Q}^\mathrm{al}$ and $K$ at $\mathfrak{q}\in\mathfrak{Q}$,
respectively, so that $\mathbb{F}_\mathfrak{q}$ has $q^2$ elements and 
$\mathbb{F}_\mathfrak{q}^\mathrm{al}$ is algebraically closed.
For each $\mathfrak{q}\in\mathfrak{Q}$ let 
\begin{displaymath}
Z_0(\mathbf{N})_\mathfrak{q}\subset Y_0(\mathbf{N})_{/\mathbb{F}_\mathfrak{q}}
\hspace{1cm}
Z_1(N,M)_\mathfrak{q}\subset Y_1(N,M)_{/\mathbb{F}_\mathfrak{q}}
\end{displaymath}
denote the subsets of supersingular points.
The points of $Z_0(\mathbf{N})_\mathfrak{q}$ all have residue 
field $\mathbb{F}_\mathfrak{q}$,
but this need not be true of $Z_1(N,M)_\mathfrak{q}$.

\begin{Def}\label{chaotic}
A subset $\mathcal{S}\subset \mathrm{Gal}(\mathbb{Q}^\mathrm{al}/K)$
is \emph{chaotic} if for any distinct $\sigma,\tau\in\mathcal{S}$,
the restriction of $\sigma\tau^{-1}$ to $\mathrm{Gal}(H[p^\infty]/K)$
is not the Artin symbol of any idele with trivial
$p$-component. 
\end{Def}

%%%%%%%%%%%%%%%%%%%%%%%%%%%%%%%%%%%%%%%%%%%%%%%%%%%%%%%%%%%%%%%%

\subsection{Simultaneous reduction}

%%%%%%%%%%%%%%%%%%%%%%%%%%%%%%%%%%%%%%%%%%%%%%%%%%%%%%%%%%%%%%%%

As $E_1$ has complex multiplication, 
any model of $E_1$ over a number field has everywhere
potentially good reduction.
Fix a finite Galois extension $F_1/K$ over which 
$E_1$ has a model with good reduction
at every prime above every rational prime $q\in\mathfrak{Q}$, 
and fix such a model.  All endomorphisms of $E_1$ are defined over
$F_1$, and hence so is the subgroup 
$\mathbf{C}_1=E_1[\mathfrak{NM}]$.
For each $g\in\mathcal{T}$ let $F_g$
be a finite  extension of $F_1$, Galois over $K$, 
over which the subgroup $X_g$ is defined.
We may then view $E_g$, $\mathbf{C}_g$, and the isogeny $f_g$ all 
as being defined over $F_g$.  Fixing these choices, we may 
reduce everything at $w_\mathfrak{q}$ to obtain a family of $\Lambda$-augmented
$\Gamma_0(\mathbf{N})$ structures over $\mathbb{F}_\mathfrak{q}^\mathrm{al}$
\begin{displaymath}
\mathrm{red}_\mathfrak{q}(E_g,\mathbf{C}_g,\Theta_g)=
(\mathrm{red}_\mathfrak{q}(E_g),
\mathrm{red}_\mathfrak{q}(\mathbf{C}_g),\mathrm{red}_\mathfrak{q}(\Theta_g)).
\end{displaymath}
We also denote by $\mathrm{red}_\mathfrak{q}(f_g)$ the reduction of the 
isogeny $f_g$.  Given an element 
$\sigma\in \mathrm{Gal}(\mathbb{Q}^\mathrm{al}/K)$ 
we may also form 
\begin{equation}
\label{twist reduction}
\mathrm{red}_\mathfrak{q}(E_g^\sigma,\mathbf{C}_g^\sigma,\Theta_g^\sigma)
=\big(\mathrm{red}_\mathfrak{q}(E_g^\sigma),
\mathrm{red}_\mathfrak{q}(\mathbf{C}_g^\sigma),
\mathrm{red}_\mathfrak{q}(\Theta_g^\sigma)\big)
\end{equation}
and 
\begin{displaymath}
\mathrm{red}_\mathfrak{q}(f_g^\sigma):
\mathrm{red}_\mathfrak{q}(E_1^\sigma)
\map{}\mathrm{red}_\mathfrak{q}(E_g^\sigma).
\end{displaymath}  
To emphasize, we regard these as
$\Lambda$-augmented $\Gamma_0(\mathbf{N})$ structures 
over $\mathbb{F}_\mathfrak{q}^\mathrm{al}$, 
regardless of the residue field of $F_g$ at $\mathfrak{q}$.  
The field of moduli of (\ref{twist reduction}) 
is $\mathbb{F}_\mathfrak{q}$; a fact
(Lemma \ref{local moduli})
whose proof we postpone until the next section.
Abbreviate
\begin{equation}\label{abbreviation}
\mathcal{Z}_\mathfrak{q}(M)=
\mathcal{A}_\Lambda^\circ(Z_1(N,M)_\mathfrak{q}; 
Y_1(N,M)_{/\mathbb{F}_\mathfrak{q}})^\Delta.
\end{equation}

Let $\Lambda[\mathcal{T}]$ denote the free $\Lambda$-module on the set
$\mathcal{T}$. For each $\sigma\in\mathrm{Gal}(\mathbb{Q}^\mathrm{al}/K)$ 
and each $\mathfrak{q}\in\mathfrak{Q}$ define the reduction map
$
\mathrm{Red}_{\sigma,\mathfrak{q}}:\Lambda[\mathcal{T}] \map{}
\mathcal{Z}_\mathfrak{q}(M)
$
by taking $L=\mathbb{F}_\mathfrak{q}$ in 
(\ref{the class}) and linearly extending  
\begin{displaymath}
\mathrm{Red}_{\sigma,\mathfrak{q}}(g) = 
F^*_{N,M}\big(\mathrm{red}_\mathfrak{q}
(E_g^\sigma),\mathrm{red}_\mathfrak{q}(\mathbf{C}_g^\sigma),
\mathrm{red}_\mathfrak{q}(\Theta_g^\sigma)\big).
\end{displaymath}
For any subset 
$\mathcal{S}\subset\mathrm{Gal}(\mathbb{Q}^\mathrm{al}/K)$ define the
simultaneous reduction map
\begin{equation}\label{SRM}
\mathrm{Red}_{\mathcal{S},\mathfrak{Q}}:\Lambda[\mathcal{T}]\map{}
\bigoplus_{(\sigma, \mathfrak{q})\in \mathcal{S}\times\mathfrak{Q}}
\mathcal{Z}_\mathfrak{q}(M)
\end{equation}
by  linearly extending $\mathrm{Red}_{\mathcal{S},\mathfrak{Q}}(g)=
\oplus_{\sigma,\mathfrak{q}} \mathrm{Red}_{\sigma,\mathfrak{q}}(g)$.
The reader may wish to skip directly to Theorem \ref{chaos},
the main result of \S \ref{S: reduction of the family}.

%%%%%%%%%%%%%%%%%%%%%%%%%%%%%%%%%%%%%%%%%%%%%%%%%%%

\subsection{Reduction at $\mathfrak{q}$}
\label{reduction}

%%%%%%%%%%%%%%%%%%%%%%%%%%%%%%%%%%%%%%%%%%%%%%%%%%%%%

Fix a $\mathfrak{q}\in\mathfrak{Q}$. Define
$S=\mathrm{End}_{\mathbb{F}_\mathfrak{q}^\mathrm{al}}
(\mathrm{red}_\mathfrak{q}(E_1))$ and $B=S\otimes\mathbb{Q}$
so that $B$ is a quaternion algebra ramified exactly at $q$ and
$\infty$, and $S$ is a maximal order in $B$.  Let 
$R\subset S$ be the subring of endomorphisms which leave
$\mathrm{red}_\mathfrak{q}(\mathbf{C}_1)$ stable, so that $R$ is a level 
$\mathbf{N}$-Eichler order in $B$. The embedding
$j:\mathcal{O}_K\map{}\mathrm{End}_{F_1}(E_1)$ 
determines an embedding which we again denote
by $j$
\begin{displaymath}
j:K\cong\mathrm{End}_{F_1}(E_1)\otimes\mathbb{Q}\hookrightarrow
\mathrm{End}_{\mathbb{F}_\mathfrak{q}^\mathrm{al}}
(\mathrm{red}_\mathfrak{q}(E_1))\otimes\mathbb{Q}\cong B,
\end{displaymath}
with $j(\mathcal{O}_K)\subset R$. For any rational prime $r$ 
and any $\mathbb{Z}$ (resp. $\mathbb{Q}$) algebra $A$, 
set $A_r=A\otimes_\mathbb{Z} \mathbb{Z}_r$
(resp. $A_r=A\otimes_\mathbb{Q} \mathbb{Q}_r$).  Let $\hat{B}$ be the 
restricted topological product $\prod'_r B_r$ with respect to
the local orders $R_r\subset B_r$, and define $\hat{K}, \hat{R},\ldots$
similarly.
The embedding $j$ induces  embeddings
$\hat{K}\hookrightarrow\hat{B}$ and $K_r\hookrightarrow B_r$
at every $r$.  We denote all of these again by $j$.

Recall that we have fixed an isomorphism
of $\mathbb{Z}_p$-modules $\mathrm{Ta}_p(E_1)\cong 
\mathbb{Z}_p^2$. As the $p$-adic Tate
modules of $E_1$ and $\mathrm{red}_\mathfrak{q}(E_1)$ 
are canonically identified as 
$\mathbb{Z}_p$-modules (and $R_p\subset B_p$ is a maximal order), 
this induces isomorphisms
\begin{equation}\label{p param}
R_p \cong M_2(\mathbb{Z}_p)
\hspace{1cm}
B_p \cong M_2(\mathbb{Q}_p).
\end{equation}
We henceforth identify $R_p^\times\cong\mathrm{GL}_2(\mathbb{Z}_p)$ and
$B_p^\times\cong\mathrm{GL}_2(\mathbb{Q}_p)$ using \emph{these} isomorphisms,
and in particular identify $\mathcal{T}$ with 
$\mathbb{Q}_p^\times R_p^\times\backslash B_p^\times$.
This gives a right action of $B_p^\times$ (and hence also of $B^\times$)
on $\mathcal{T}$.
The group $R_\ell^\times$ acts on 
$\mathrm{Ta}_\ell(\mathrm{red}_\mathfrak{q}(E_1))$
on the left, almost
by definition, and we denote by $\rho_\mathfrak{q}$ the action of
$R_\ell^\times$ on 
$\mathcal{A}_\Lambda(\mathrm{red}_\mathfrak{q}(E_1))$ 
obtained by taking symmetric powers,
with the understanding that $R_\ell^\times$ acts trivially on the twist
$\Lambda(1-k)$. Writing $\det$ for the reduced norm on 
$B^\times$, $B_\ell^\times$, and so on,  we also define 
$\rho_\mathfrak{q}^*=\rho_\mathfrak{q}\otimes\det^{1-k}$, 
and note that the center
$\mathbb{Z}_\ell^\times\subset B_\ell^\times$ 
acts trivially under $\rho_\mathfrak{q}^*$. The group 
$\Gamma_\mathfrak{q}= R[1/p]^\times$
acts on  $\mathcal{A}_\Lambda(\mathrm{red}_\mathfrak{q}(E_1))$ 
through $\rho_\mathfrak{q}$ or $\rho_\mathfrak{q}^*$
by the inclusion $\Gamma_\mathfrak{q}\hookrightarrow R_\ell^\times$.

Fix a $\sigma\in\mathrm{Gal}(\mathbb{Q}^\mathrm{al}/K)$ whose restriction
to $K^\mathrm{ab}$ (the maximal abelian extension of $K$)
is equal to the Artin symbol of a finite idele
$\hat{\sigma}\in\hat{K}^\times$.
Let $b_{\sigma,\mathfrak{q}}\in  B^\times$ be such that the $r$-component of
  $j(\hat{\sigma})b_{\sigma,\mathfrak{q}}$ lies in $R_r^\times$ for all primes
  $r\not=p$, and let $\alpha_{\sigma,\mathfrak{q}}\in R_\ell^\times$ and 
$\beta_{\sigma,\mathfrak{q}}\in B_p^\times$ be the
$\ell$ and $p$ components, respectively,  of 
$j(\hat{\sigma})b_{\sigma,\mathfrak{q}}\in\hat{B}^\times$.

\begin{Prop}\label{first action}
Fix $g,h\in\mathcal{T}$. There is a $\gamma\in\Gamma_\mathfrak{q}$ such that 
$g \beta_{\sigma,\mathfrak{q}} = h\gamma\in\mathcal{T}$,
if and only if there is an isomorphism of $\Gamma_0(\mathbf{N})$ structures
over $\mathbb{F}_\mathfrak{q}^\mathrm{al}$
\begin{displaymath}
\phi:\mathrm{red}_\mathfrak{q}(E_g^\sigma,\mathbf{C}_g^\sigma)\cong 
\mathrm{red}_\mathfrak{q}(E_h,\mathbf{C}_h).
\end{displaymath}
If these equivalent conditions hold then $\phi$ may be chosen so that
\begin{displaymath}
\phi(\mathrm{red}_\mathfrak{q}(\Theta_g^\sigma))= 
\omega_\mathrm{cyc}^{1-k}(\sigma)\cdot 
\mathrm{red}_\mathfrak{q}(f_h)
\big(\rho_\mathfrak{q}(\gamma \alpha_{\sigma,\mathfrak{q}}^{-1})
\mathrm{red}_\mathfrak{q}(\Theta_1)\big),
\end{displaymath}
where $\omega_\mathrm{cyc}$ is the $\ell$-adic cyclotomic
character and $\gamma\in\Gamma_\mathfrak{q}$ has the property that
there are cyclic (in the sense of Definition \ref{cyclic})
lifts $c(g), c(h)\in B_p^\times$ of $g$ and $h$
satisfying 
$c(g)\beta_{\sigma,\mathfrak{q}}=c(h)\gamma$ in 
$R_p^\times\backslash B_p^\times$.
\end{Prop}

\begin{proof} 
For any $g\in\mathcal{T}$ there is an isomorphism of 
$\Gamma_0(\mathbf{N})$ structures over $\mathbb{F}_\mathfrak{q}^\mathrm{al}$
\begin{equation}\label{first action i}
\mathrm{red}_\mathfrak{q}(E_g^\sigma, \mathbf{C}_g^\sigma)\cong 
\mathrm{red}_\mathfrak{q}(E_{g\beta_{\sigma,\mathfrak{q}}},
\mathbf{C}_{g\beta_{\sigma,\mathfrak{q}}}).
\end{equation}
This is exactly the calculation performed in \cite[\S 3.3]{cornut02}.
On the other hand, by the parametrization of 
$Z_0(\mathbf{N})_\mathfrak{q}$ given 
in \cite[\S 2.3]{cornut02} there is an isomorphism
\begin{equation}\label{first action ii}
\mathrm{red}_\mathfrak{q}(E_{g\beta_{\sigma,\mathfrak{q}}},
\mathbf{C}_{g\beta_{\sigma,\mathfrak{q}}})
\cong
\mathrm{red}_\mathfrak{q}(E_h,\mathbf{C}_h)
\end{equation}
if and only if $g\beta_{\sigma,\mathfrak{q}}$ and $h$ lie in the same orbit
under the right action of $\Gamma_\mathfrak{q}$ on $\mathcal{T}$.  
This proves the first claim.
The proof of the second claim follows from an examination of 
the isomorphisms (\ref{first action i}) and (\ref{first action ii}),
and we give a sketch.
The isomorphisms (\ref{first action i}) and (\ref{first action ii}), 
disregarding the $\Gamma_0(\mathbf{N})$ structure,
arise from isomorphisms (again, see \cite[\S 3.3]{cornut02})
\begin{equation}
\mathrm{red}_\mathfrak{q}(E^\sigma_g)  \cong  
\mathrm{Hom}_R( R\cdot c(g)j(\hat{\sigma}),\mathrm{red}_\mathfrak{q}(E_1)) 
\label{first action iii}
\end{equation}
\begin{equation}
\mathrm{red}_\mathfrak{q}(E_h)  \cong 
\mathrm{Hom}_R( R\cdot c(h),\mathrm{red}_\mathfrak{q}(E_1))
\label{first action iv}
\end{equation}
of functors on $\mathbb{F}_\mathfrak{q}^\mathrm{al}$-schemes, where  
$\mathrm{Hom}_R$ means homomorphisms of left $R$-modules, and
$c(g)$ and $c(h)$ are viewed as elements of 
$\hat{B}^\times$ with trivial components away from $p$.
The map $x\mapsto x b_{\sigma,\mathfrak{q}}\gamma^{-1}$.
induces an isomorphism of left $R$-submodules of $\hat{B}$
\begin{displaymath}
R\cdot c(g)j(\hat{\sigma}) \map{\cdot b_{\sigma,\mathfrak{q}}}
R\cdot c(g)j(\hat{\sigma})b_{\sigma,\mathfrak{q}} =
R\cdot c(g)\beta_{\sigma,\mathfrak{q}} =
R\cdot c(h)\gamma \map{\cdot \gamma^{-1}}
R\cdot c(h)
\end{displaymath}
and so identifies 
$\mathrm{red}_\mathfrak{q}(E^\sigma_g)\cong \mathrm{red}_\mathfrak{q}(E_h)$
and 
\begin{equation}
\label{R modules}
\mathrm{Hom}_{R_\ell}(R_\ell\cdot j(\hat{\sigma})_\ell , 
\mathrm{Ta}_\ell(\mathrm{red}_\mathfrak{q}(E_1))
\cong
\mathrm{Hom}_{R_\ell}(R_\ell ,  
\mathrm{Ta}_\ell(\mathrm{red}_\mathfrak{q}(E_1))
\end{equation}
By the main theorem of complex multiplication, the isomorphism 
(\ref{first action iii}) may be chosen so that the induced isomorphism
\begin{displaymath}
\mathrm{Ta}_\ell(E_1)\map{f_g}\mathrm{Ta}_\ell(E_g)\map{\sigma}
\mathrm{Ta}_\ell(E_g^\sigma)
\cong \mathrm{Hom}_{R_\ell}(R_\ell \cdot j(\hat{\sigma})_\ell, 
\mathrm{Ta}_\ell(\mathrm{red}_\mathfrak{q}(E_1))
\end{displaymath}
takes $t\in\mathrm{Ta}_\ell(E_1)\cong 
\mathrm{Ta}_\ell(\mathrm{red}_\mathfrak{q}(E_1))$
to the $R_\ell$-linear map determined by 
$j(\hat{\sigma})_\ell \mapsto t $. 
The isomorphism
(\ref{R modules}) takes $j(\hat{\sigma})_\ell \mapsto t$ to 
$\alpha_{\sigma,\mathfrak{q}}\gamma^{-1}\mapsto t$.  Under
(\ref{first action iv}) this latter map corresponds to 
$\mathrm{red}_{\mathfrak{q}}(f_h) (\gamma\alpha_{\sigma,\mathfrak{q}}^{-1} t)\
\in\mathrm{Ta}_\ell(\mathrm{red}_\mathfrak{q}(E_h))
$
This shows that the composition
\begin{displaymath}
\mathrm{Ta}_\ell(E_1)\map{f_g}\mathrm{Ta}_\ell(E_g)\map{\sigma}
\mathrm{Ta}_\ell(E_g^\sigma)\cong
\mathrm{Ta}_\ell(\mathrm{red}_\mathfrak{q}(E_g^\sigma))
\cong \mathrm{Ta}_\ell(\mathrm{red}_\mathfrak{q}(E_h))
\end{displaymath}
is given by $t\mapsto \mathrm{red}_\mathfrak{q}(f_h)( 
\gamma\alpha_{\sigma,\mathfrak{q}}^{-1} t)$.
The proposition now follows by taking symmetric
powers and twisting by $\Lambda(1-k)$.
\end{proof}

\begin{Cor}\label{second action}
Let $\sigma$ and  $\beta_{\sigma,\mathfrak{q}}$ 
be as in Proposition \ref{first action}.  
For each $h\in\mathcal{T}$ there is a 
$\varpi_{\sigma,\mathfrak{q},h}\in
\mathcal{A}_\Lambda(\mathrm{red}_\mathfrak{q}(E_1))$
with the property that for any  
$g\in h\Gamma_\mathfrak{q}\beta_{\sigma,\mathfrak{q}}^{-1}\subset\mathcal{T}$
there exists an isomorphism of 
$\Lambda$-augmented $\Gamma_0(N)$ 
structures over $\mathbb{F}_\mathfrak{q}^\mathrm{al}$
\begin{equation}\label{monodromy}
\mathrm{red}_\mathfrak{q}
(E_g^\sigma,\mathbf{C}_g^\sigma,\deg(g)^{1-k}\cdot \Theta_g^\sigma)
\cong
\big(\mathrm{red}_\mathfrak{q}(E_h), \mathrm{red}_\mathfrak{q}(\mathbf{C}_h), 
\mathrm{red}_\mathfrak{q}(f_h)( 
\rho^*_\mathfrak{q}(\gamma)\varpi_{\sigma,\mathfrak{q},h} )\big)
\end{equation}
where $\gamma\in\Gamma_\mathfrak{q}$ is any element with 
$g\beta_{\sigma,\mathfrak{q}}= h\gamma$ in $\mathcal{T}$.
\end{Cor}

\begin{proof}
Suppose we have an equality $g=h\gamma\beta^{-1}_{\sigma,\mathfrak{q}}$
in $\mathcal{T}$ with $g,h\in\mathcal{T}$ and $\gamma\in\Gamma_\mathfrak{q}$.
Fix cyclic lifts $c(g)$ and $c(h)$ of $g$ and $h$, respectively, 
to $B_p^\times$, and choose $\gamma_0\in\gamma\cdot \mathbb{Z}[1/p]^\times$
so that $c(g)\beta_{\sigma,\mathfrak{q}}=c(h)\gamma_0$ in 
$R_p^\times\backslash B_p^\times$.  Using Proposition \ref{first action}
and the fact that 
$\rho_\mathfrak{q}^*(\gamma_0)=\rho_\mathfrak{q}^*(\gamma)$, one checks
directly that (\ref{monodromy}) holds with
\begin{displaymath}
\varpi_{\sigma,\mathfrak{q},h}=
\Big(\deg(g)\omega_\mathrm{cyc}(\sigma)
\det(\gamma_0^{-1})\det(\alpha_{\sigma,\mathfrak{q}})\Big)^{1-k}
\rho^*_\mathfrak{q}(\alpha_{\sigma,\mathfrak{q}}^{-1})
\mathrm{red}_\mathfrak{q}(\Theta).
\end{displaymath}
As $\deg(g)\det(\gamma_0)^{-1}
=\deg(h)p^{-\mathrm{ord_p}\det(\beta_{\sigma,\mathfrak{q}})}$
depends on $h$ but not on $g$, the same is true of 
$\varpi_{\sigma,\mathfrak{q},h}$.
\end{proof}

%%%%%%%%%%%%%%%%%%%%%%%%%%%%%%%%%%%%%%%%%%%%%%%%%%%%%%%%%%%%%%%

\subsection{Vatsal's lemma}

%%%%%%%%%%%%%%%%%%%%%%%%%%%%%%%%%%%%%%%%%%%%%%%%%%%%%%%%%%%%%%%%

Fix a subset $\mathcal{S}\subset\mathrm{Gal}(\mathbb{Q}^\mathrm{al}/K)$.
For each $\mathfrak{q}\in\mathfrak{Q}$ and each $\sigma\in\mathcal{S}$
let $\beta_{\sigma,\mathfrak{q}}\in B_p^\times$ be as in Proposition 
\ref{first action}.  The quaternion algebra $B$ depends on $\mathfrak{q}$,
but using the isomorphisms of (\ref{p param}) we identify
$B_p^\times\cong\mathrm{GL}_2(\mathbb{Q}_p)$ and view both 
$\beta_{\sigma,\mathfrak{q}}$
and $\Gamma_\mathfrak{q}$ as living in $\mathrm{GL}_2(\mathbb{Q}_p)$ 
under this identification.

\begin{Lem}\label{smaller subgroups}
For each $\mathfrak{q}\in\mathfrak{Q}$
there is a finite index subgroup 
$\Gamma_\mathfrak{q}^*\subset\Gamma_\mathfrak{q}$ 
containing $\mathbb{Z}[1/p]^\times$ such that
$\det(\Gamma_\mathfrak{q}^*)=p^\mathbb{Z}$ 
and the restriction of 
$\rho^*_\mathfrak{q}$ to $\Gamma_\mathfrak{q}^*$ is trivial.
\end{Lem}

\begin{proof}
Define a subgroup $U=\prod U_r\subset \hat{B}^\times$ by
\begin{displaymath}
U_r=\left\{\begin{array}{ll}
\mathrm{Ker}\big(\rho_\mathfrak{q}^*: R_\ell^\times\map{}
\mathrm{Aut}(\mathcal{A}_\Lambda(\mathrm{red}_\mathfrak{q}(E_1)))\big) &
\mathrm{if\ } r=\ell \\
B_p^\times & \mathrm{if\ }r=p\\
R_r^\times &\mathrm{else}
\end{array}\right.
\end{displaymath}
and let $\Gamma_\mathfrak{q}^*=B^\times\cap U\subset \hat{B}^\times$.
Then $\Gamma_\mathfrak{q}^*\subset\Gamma_\mathfrak{q}$ is
 exactly the kernel of 
$\rho_\mathfrak{q}^*$ restricted to $\Gamma_\mathfrak{q}$.  We must show that
$\Gamma_\mathfrak{q}^*$ contains an element of norm $p$.  By 
\cite[Theoreme III.4.1]{vigneras} there is a $b_0\in B^\times$
of norm $p$.  Let $x=(x_r)\in U$ be an element of norm 
$p\in\hat{\mathbb{Q}}^\times$.
By strong approximation \cite[Theoreme III.4.3]{vigneras}
the norm one element $b_0^{-1}x\in\hat{B}^\times$
may be written in the form $b_1 y u=b_0^{-1}x$ for some norm one 
elements $b_1\in B^\times$, $y\in B_p^\times $, and $u\in U$. 
Then $b_0 b_1$ has norm $p$ and is contained in $\Gamma_\mathfrak{q}^*$.
\end{proof}

\begin{Prop}\label{tree surjection}
For each $\mathfrak{q}\in\mathfrak{Q}$ let $\Gamma_\mathfrak{q}^*$
be as in Lemma \ref{smaller subgroups}, and for each 
$(\sigma,\mathfrak{q})\in\mathcal{S}\times\mathfrak{Q}$ set
\begin{displaymath}
\Gamma_{\sigma,\mathfrak{q}}^*=
\beta_{\sigma,\mathfrak{q}}\Gamma_\mathfrak{q}^* 
\beta_{\sigma,\mathfrak{q}}^{-1}\subset 
\mathrm{GL}_2(\mathbb{Q}_p).
\end{displaymath}
If $\mathcal{S}$ is chaotic then the quotient map
$\mathcal{T}\map{}\prod_{(\sigma,\mathfrak{q})\in\mathcal{S}\times\mathfrak{Q}}
\mathcal{T}/\Gamma^*_{\sigma,\mathfrak{q}}$
is surjective.  
\end{Prop}

\begin{proof}
Let $\tilde{\Gamma}^*_{\sigma,\mathfrak{q}}$ be the image of 
$\Gamma_{\sigma,\mathfrak{q}}^*$ in $\mathrm{PGL}_2(\mathbb{Q}_p)$ and let
$\tilde{\Gamma}^{*,1}_{\sigma,\mathfrak{q}}$ the intersection of 
$\tilde{\Gamma}^*_{\sigma,\mathfrak{q}}$ with
$\mathrm{PSL}_2(\mathbb{Q}_p)$.  Then 
$\tilde{\Gamma}^{*,1}_{\sigma,\mathfrak{q}}$
is discrete and cocompact by \cite[p.104]{vigneras}, 
and these subgroups are pairwise 
non-commensurable as $(\sigma,\mathfrak{q})$ varies by
\cite[Proposition 3.7]{cornut02}.
By Vatsal's application of a theorem of Ratner (see
\cite[Proposition 3.11]{cornut02} or \cite[Lemma 5.10]{vatsal02}),
the natural map
\begin{displaymath}
\mathrm{PSL}_2(\mathbb{Q}_p) \map{}
\prod_{(\sigma,\mathfrak{q})\in\mathcal{S}\times\mathfrak{Q}}
\mathrm{PSL}_2(\mathbb{Z}_p)\backslash \mathrm{PSL}_2(\mathbb{Q}_p)/
\tilde{\Gamma}^{*,1}_{\sigma,\mathfrak{q}}
\end{displaymath}
is surjective, and the proposition follows as in 
\cite[Proposition 3.4]{cornut02}.
\end{proof}

%%%%%%%%%%%%%%%%%%%%%%%%%%%%%%%%%%%%%%%%%%%%%%%%%%%%%%%%%%%%%%%

\subsection{Surjectivity of the reduction map}
\label{chaotic theorem}

%%%%%%%%%%%%%%%%%%%%%%%%%%%%%%%%%%%%%%%%%%%%%%%%%%%%%%%%%%%%%

Assume  $\mathcal{S}\subset\mathrm{Gal}(\mathbb{Q}^\mathrm{al}/K)$ is finite
and chaotic.

\begin{Prop}\label{reduction I}
Fix $(\sigma',\mathfrak{q}')\in\mathcal{S}\times\mathfrak{Q}$,
$\gamma_0,\gamma_1\in\Gamma_\mathfrak{q}$, and $h\in\mathcal{T}$.  
There exist $g_0,g_1\in\mathcal{T}$ such that
\begin{equation*}
\deg(g_0)^{1-k}\mathrm{Red}_{\mathcal{S},\mathfrak{Q}}(g_0)
-
\deg(g_1)^{1-k}\mathrm{Red}_{\mathcal{S},\mathfrak{Q}}(g_1)
\in \bigoplus_{(\sigma,\mathfrak{q})\in\mathcal{S}\times\mathfrak{Q}}
\mathcal{Z}_\mathfrak{q}(M)
\end{equation*}
has trivial components except at the summand 
$(\sigma,\mathfrak{q})=(\sigma',\mathfrak{q}')$,
at which the component is equal to 
\begin{displaymath}
F^*_{N,M}\big(\mathrm{red}_\mathfrak{q}(E_h),
\mathrm{red}_\mathfrak{q}(\mathbf{C}_h),
\mathrm{red}_\mathfrak{q}(f_h)\big(\rho^*_\mathfrak{q}
(\gamma_0)\varpi_{\sigma,\mathfrak{q},h}-
\rho^*_\mathfrak{q}(\gamma_1)\varpi_{\sigma,\mathfrak{q},h}\big)\big),
\end{displaymath}
where $\varpi_{\sigma,\mathfrak{q},h}\in\mathcal{A}_\Lambda
(\mathrm{red}_{\mathfrak{q}}(E_1))$ 
is the element of Corollary \ref{second action}.
\end{Prop}

\begin{proof}
For each $i\in\{0,1\}$ Proposition \ref{tree surjection} allows us to choose a 
$g_i\in\mathcal{T}$ such that the reduction map
$\mathcal{T}\map{} \mathcal{T}/\Gamma^*_{\sigma,\mathfrak{q}}$ takes
\begin{displaymath}
g_i\mapsto\left\{
\begin{array}{ll}
h\gamma_i\beta_{\sigma,\mathfrak{q}}^{-1}\Gamma^*_{\sigma,\mathfrak{q}} 
= h\gamma_i\Gamma^*_{\mathfrak{q}}\beta^{-1}_{\sigma,\mathfrak{q}}
&\mathrm{if\ }(\sigma,\mathfrak{q})=
(\sigma',\mathfrak{q}') \\
h\beta_{\sigma,\mathfrak{q}}^{-1}\Gamma^*_{\sigma,\mathfrak{q}} =
h\Gamma^*_{\mathfrak{q}}\beta^{-1}_{\sigma,\mathfrak{q}}
& \mathrm{if\ } (\sigma,\mathfrak{q})\not=(\sigma',\mathfrak{q}')
\end{array}\right.
\end{displaymath}
for every $(\sigma,\mathfrak{q})\in\mathcal{S}\times\mathfrak{Q}$.
By Corollary \ref{second action} we have
\begin{displaymath}
\mathrm{red}_\mathfrak{q}(E_{g_0}^\sigma, \mathbf{C}_{g_0}^\sigma, 
\deg(g_0)^{1-k}\Theta^\sigma_{g_0})
\cong
\mathrm{red}_\mathfrak{q}(E_{g_1}^\sigma, \mathbf{C}_{g_1}^\sigma, 
\deg(g_1)^{1-k}\Theta^\sigma_{g_1})
\end{displaymath}
as a $\Lambda$-augmented $\Gamma_0(\mathbf{N})$ structure over 
$\mathbb{F}^\mathrm{al}_\mathfrak{q}$
whenever $(\sigma,\mathfrak{q})\not= (\sigma',\mathfrak{q}')$, while 
\begin{displaymath}
\mathrm{red}_{\mathfrak{q}}(E_{g_i}^{\sigma}, 
\mathbf{C}_{g_i}^{\sigma}, \deg(g_i)^{1-k}\Theta^{\sigma}_{g_i})
\cong \big(\mathrm{red}_{\mathfrak{q}}(E_h),\mathrm{red}_{\mathfrak{q}}(\mathbf{C}_h),
\mathrm{red}_{\mathfrak{q}}(f_h)(\rho^*_{\mathfrak{q}}(\gamma_i)
\varpi_{\sigma,\mathfrak{q},h})\big)
\end{displaymath}
if $(\sigma,\mathfrak{q})=(\sigma',\mathfrak{q}')$.
The proposition  is now immediate from the definition (\ref{SRM}) of  
$\mathrm{Red}_{\mathcal{S},\mathfrak{Q}}$.
\end{proof}

\begin{Lem}\label{irreducible}
Fix $\mathfrak{q}\in\mathfrak{Q}$ and suppose $\Lambda=\mathbb{Z}/\ell\mathbb{Z}$. 
Then $\mathcal{A}_\Lambda(\mathrm{red}_\mathfrak{q}(E_1))$
has no proper, nonzero $\Lambda$-submodules 
which are stable under $\rho^*_\mathfrak{q}(\Gamma_\mathfrak{q})$.
\end{Lem}

\begin{proof}
Fix a $\mathbb{Z}_\ell$-basis for the $\ell$-adic Tate module of 
$\mathrm{red}_\mathfrak{q}(E_1)$, so that $R_\ell^\times$ is identified with 
$\mathrm{GL}_2(\mathbb{Z}_\ell)$.
Let $\Gamma_\mathfrak{q}^1$ and $R_\ell^{\times,1}$ denote the norm one elements
of $\Gamma_\mathfrak{q}$ and $R_\ell^\times$, respectively.  
Then $\mathcal{A}_\Lambda(\mathrm{red}_\mathfrak{q}(E_1))$ is 
identified with $\mathrm{Sym}^{2k-2}\Lambda^2$ and the action of 
$\rho_\mathfrak{q}^*$ restricted to $\Gamma_\mathfrak{q}^1$ is through
\begin{displaymath}\Gamma_\mathfrak{q}^1\map{}
R_\ell^{\times,1}\map{}\mathrm{SL}_2(\mathbb{Z}_\ell)
\map{}\mathrm{SL}_2(\Lambda).\end{displaymath}
Using strong approximation \cite[Theoreme III.4.3]{vigneras}
one may show that the first arrow has dense image, and so the
composition is surjective.  
By the assumption $\ell>2k-2$,
$\mathrm{Sym}^{2k-2}\Lambda^2$ has no proper, nonzero submodules stable
under the action of $\mathrm{SL}_2(\Lambda)$.
\end{proof}

\begin{Thm}\label{chaos}
Let $\mathcal{S}\subset\mathrm{Gal}(\mathbb{Q}^\mathrm{al}/K)$ be finite and chaotic,
and suppose $\Lambda=\mathbb{Z}/\ell\mathbb{Z}$.  Then
the simultaneous reduction map (\ref{SRM}) is surjective.
\end{Thm}

\begin{proof}
Fix $(\sigma',\mathfrak{q}')\in \mathcal{S}\times\mathfrak{Q}$ 
and a supersingular point $z\in Z_0(\mathbf{N})_{\mathfrak{q}'}$.
Let $Z\subset Z_1(N,M)_{\mathfrak{q}'}$ be the set of closed points lying
above $z$.  We will show that the image of (\ref{SRM})
contains the submodule
\begin{equation}\label{chaos submodule}
\mathcal{A}_\Lambda^\circ(Z;Y_1(N,M)_{/\mathbb{F}_\mathfrak{q}'})^\Delta
\subset
\bigoplus_{(\sigma, \mathfrak{q})\in \mathcal{S}\times\mathfrak{Q}}
\mathcal{A}_\Lambda^\circ(Z_1(N,M)_\mathfrak{q}; 
Y_1(N,M)_{/\mathbb{F}_\mathfrak{q}})^\Delta
\end{equation}
supported in the $(\sigma',\mathfrak{q}')$ component.
The parametrization \cite[\S 2.3]{cornut02} shows that
the map $\mathcal{T}\mapsto Z_0(\mathbf{N})_{\mathfrak{q}'}$ defined by
$h\mapsto \mathrm{red}_{\mathfrak{q}'}(E_h,\mathbf{C}_h)$ establishes a bijection
$\mathcal{T}/\Gamma_{\mathfrak{q}'}\cong Z_0(\mathbf{N})_{\mathfrak{q}'}.$
Thus we may fix an $h\in\mathcal{T}$ such that the supersingular 
$\Gamma_0(\mathbf{N})$ structure
$\mathrm{red}_{\mathfrak{q}'}(E_h,\mathbf{C}_h)$ corresponds to the point $z$.
For any $g\in\mathcal{T}$, $\mathrm{red}_{\mathfrak{q}'}(\Theta_g^{\sigma'})\not=0$ 
(from the construction one sees that
$\Theta_1\not=0$, and $\ell\not=p$ implies that
$f_g:\mathcal{A}_\Lambda(E_1)\map{}\mathcal{A}_\Lambda(E_g)$
is an isomorphism).  It follows that $\varpi_{\sigma',\mathfrak{q}',h}\not=0$.
By Lemma \ref{irreducible} we may choose a $\gamma_1\in\Gamma_{\mathfrak{q}'}$
such that 
\begin{displaymath}
\pi\stackrel{\mathrm{def}}{=}\rho_{\mathfrak{q}'}^*
(\gamma_1)\varpi_{\sigma',\mathfrak{q}',h}-
\varpi_{\sigma',\mathfrak{q}',h}\in 
\mathcal{A}_\Lambda(\mathrm{red}_{\mathfrak{q}'}(E_1))
\end{displaymath}
is nonzero.  Again by Lemma \ref{irreducible}, choose
$\gamma^{(0)},\ldots,\gamma^{(n)}\in\Gamma_{\mathfrak{q}'}$ such that 
the elements
$\rho_{\mathfrak{q}'}^*(\gamma^{(i)})\pi$, $0\le i\le n$, generate 
$\mathcal{A}_\Lambda(\mathrm{red}_{\mathfrak{q}'}(E_1))$. 
Set $\gamma_0^{(i)}=\gamma^{(i)}\gamma_1$ and
let $g_0^{(i)}, g_1$ be as in Proposition \ref{reduction I},
so that
\begin{displaymath}
\deg(g_0^{(i)})^{1-k}
\mathrm{Red}_{\mathcal{S},\mathfrak{Q}}(g_0^{(i)})
-
\deg(g_1)^{1-k}
\mathrm{Red}_{\mathcal{S},\mathfrak{Q}}(g_1)
\end{displaymath}
has trivial components except at the summand 
$(\sigma,\mathfrak{q})=(\sigma',\mathfrak{q}')$, where the component is equal to 
\begin{equation}\label{chaos element}
F_{N,M}^*\big(\mathrm{red}_{\mathfrak{q}'}(E_h),
\mathrm{red}_{\mathfrak{q}'}(\mathbf{C}_h),
\mathrm{red}_{\mathfrak{q}'}(f_h)(\rho^*_{\mathfrak{q}'}(\gamma^{(i)})\pi)\big).
\end{equation}
As $i$ varies the elements 
$\mathrm{red}_{\mathfrak{q}'}(f_h)(\rho^*_{\mathfrak{q}'}(\gamma^{(i)})\pi)$
generate $\mathcal{A}_\Lambda(\mathrm{red}_{\mathfrak{q}'}(E_h))$, and the 
elements (\ref{chaos element}) generate the submodule
(\ref{chaos submodule}).
\end{proof}

%%%%%%%%%%%%%%%%%%%%%%%%%%%%%%%%%%%%%%%%%%%%%%%%%%%%%%%%%%%%%%%%%%
%%%%%%%%%%%%%%%%%%%%%%%%%%%%%%%%%%%%%%%%%%%%%%%%%%%%%%%%%%%%%%%%%%

%\appendix
\section{Augmented theorems of Deuring, Ihara, and Ribet}
\label{append}

%%%%%%%%%%%%%%%%%%%%%%%%%%%%%%%%%%%%%%%%%%%%%%%%%%%%%%%%%%%%%%%%
%%%%%%%%%%%%%%%%%%%%%%%%%%%%%%%%%%%%%%%%%%%%%%%%%%%%%%%%%%%%%%%%%%

Let $q\nmid \mathbf{N}$ be a rational prime and let 
$\mathbb{F}$ ($=F$ when we refer to the
notions of \S \ref{S: kummer}) be a field of $q^2$ elements with 
algebraic closure $\mathbb{F}^\mathrm{al}$.
Unless specified otherwise, all geometric objects (e.g.
$Y_1(N)$, $Y_1(N,M)$,\ldots)  are defined over $\mathrm{Spec}(\mathbb{F})$.
Let $\Lambda=\mathbb{Z}/\ell\mathbb{Z}$ for some prime $\ell$ and assume that
$\ell$ does not divide $N q$.
Let $\mathcal{L}_\Lambda$ be the locally constant constructible sheaf on 
$Y_1(N)$ defined by (\ref{sheaf def}).
Denote by
\begin{displaymath}
Z_1(N)\subset Y_1(N)
\hspace{1cm}
Z_1(N,M)\subset Y_1(N,M)
\end{displaymath}
the subsets of supersingular closed points.

%%%%%%%%%%%%%%%%%%%%%%%%%%%%%%%%%%%%%%%%%%%%%%%%%%%%%%%%%%%%%%%%%%%%5

\subsection{Fields of moduli}

%%%%%%%%%%%%%%%%%%%%%%%%%%%%%%%%%%%%%%%%%%%%%%%%%%%%%%%%%%%%%%%%%%%%%

We need a slight generalization of the well-known theorem of Deuring that all
supersingular points on $Y_0(N)$ have residue degree one.

\begin{Lem}\label{local moduli}
Let $E$ be a supersingular elliptic curve over $\mathbb{F}^\mathrm{al}$, let
$C\subset E[N]$ be a cyclic subgroup of order $N$, and let
$\Theta$ be any element of $\mathcal{A}_\Lambda(E)$.  The field of moduli of the
$\Lambda$-augmented $\Gamma_0(N)$ structure $(E,C,\Theta)$ is $\mathbb{F}$.
\end{Lem}

\begin{proof}
As $E$ is supersingular, its $j$-invariant lies in
$\mathbb{F}$.  Let $A$ be an elliptic curve over $\mathbb{F}$ with the same 
$j$-invariant as $E$, and let $\mathrm{Fr}\in\mathrm{End}_{\mathbb{F}}(A)$ 
be the degree $q^2$ (relative) Frobenius.
If $\mathrm{Fr}\in\mathbb{Z}$, then $\mathrm{Fr}$ commutes with all elements of
$\mathrm{End}_{\mathbb{F}^\mathrm{al}_\mathfrak{q}}(A)$, and so
\begin{equation}\label{endomorphism ring}
\mathrm{End}_{\mathbb{F}}(A)=\mathrm{End}_{\mathbb{F}^\mathrm{al}}(A).
\end{equation}
If $\mathrm{Fr}\not\in\mathbb{Z}$ then $\mathrm{Fr}$ generates a quadratic imaginary
subfield $L$ of the definite quaternion algebra (ramified exactly 
at $q$ and $\infty$) $\mathrm{End}_{\mathbb{F}^\mathrm{al}}(A)\otimes\mathbb{Q}$,
and $q$ is nonsplit in $L$.  As $\mathrm{Fr}$ has degree 
 $q^2$ we must have
$\mathrm{Fr}=\zeta^{-1}q$ for some root of unity $\zeta\in L$, and in fact
$\zeta$ belongs to $L\cap \mathrm{End}_{\mathbb{F}}(A)$ (this 
follows from the fact \cite[Corollary 12.3.5]{katz-mazur}
that $\mathrm{Fr}$ and $[q]$ have the same scheme-theoretic kernel,
and so there is a factorization $[q]=\zeta\circ\mathrm{Fr}$
for some automorphism $\zeta$ of $A$).  
Replacing $A$ by its twisted form
corresponding to the cocycle sending the relative Frobenius 
$\sigma\in\mathrm{Gal}(\mathbb{F}^\mathrm{al}/\mathbb{F})$ to 
$\zeta\in\mathrm{Aut}_{\mathbb{F}}(E)$,
a simple calculation shows that (\ref{endomorphism ring})
holds.  Then $\mathrm{Fr}$ is a central element of 
$\mathrm{End}_{\mathbb{F}^\mathrm{al}}(A)$, and so $\mathrm{Fr}=[\pm q]$.

With this choice of $A$,  
$\mathrm{Gal}(\mathbb{F}^\mathrm{al}/\mathbb{F})$ 
acts trivially
on $\mathcal{A}_\Lambda(A)$  and the triple
$(A,C_A,\Theta_A)$ is defined over $\mathbb{F}$ for \emph{any}
cyclic order $N$ subgroup $C_A\subset A(\mathbb{F}^\mathrm{al})$ 
and \emph{any} $\Theta_A\in\mathcal{A}_\Lambda(A)$.
Over $\mathbb{F}^\mathrm{al}$ we may fix an isomorphism $f:E\cong A$
and set $C_A=f(C)$ and $\Theta_A=f(\Theta)$.
Then $(E,C,\Theta)$ and
$(A,C_A,\Theta_A)$ are isomorphic (over $\mathbb{F}^\mathrm{al}$) 
and so have the same field of moduli  $\mathbb{F}$.
\end{proof}

%%%%%%%%%%%%%%%%%%%%%%%%%%%%%%%%%%%%%%%%%%%%%%%%%%%%%%%%%%%

\subsection{Ihara's theorem}
\label{SS:ihara}

%%%%%%%%%%%%%%%%%%%%%%%%%%%%%%%%%%%%%%%%%%%%%%%%%%%%%%%%%%%%

We now recall a theorem of Ihara \cite{ihara} and derive some
consequences; our exposition of Ihara's theorem 
is influenced by the discussion of 
\cite[Chaptire 7]{cornut-thesis}.
For each integer $m$ prime to $q$ set
\begin{displaymath}
\underline{\mu}^*_m=\mathrm{Spec}(\mathbb{F}[X]/\Phi_m(X))
\hspace{1cm}
\underline{\mu}_m=\mathrm{Spec}(\mathbb{F}[X]/(X^m-1)),
\end{displaymath}
where $\Phi_m(X)$ is the $m^\mathrm{th}$ cyclotomic polynomial.
Let $Y(m)$ be the affine modular curve classifying ``naive'' level
$m$ structures in the sense of \cite{katz-mazur} on elliptic
curves over $\mathbb{F}$-schemes.  Thus $Y(m)$ is a fine moduli space if $m>2$,
and for all $m$ (prime to $q$) the Weil pairing provides a canonical
map $Y(m)\map{}\underline{\mu}^*_m$
of $\mathbb{F}$-schemes.  Fix a topological generator
\begin{displaymath}
\zeta=\varprojlim_{(m,q)=1}\zeta_m\in 
\varprojlim_{(m,q)=1}\underline{\mu}_m(\mathbb{F}^\mathrm{al}).
\end{displaymath}
For each $m$ there is a map $\mathrm{Spec}(\mathbb{F}[\zeta_m])\map{}
\underline{\mu}_m^*$
determined by the  map $X\mapsto \zeta_m$ on $\mathbb{F}$-algebras.  Define 
\begin{displaymath}
Y_\zeta(m)= Y(m)\times_{\underline{\mu}_m^*} 
\mathrm{Spec}(\mathbb{F}[\zeta_m]),
\end{displaymath}
a smooth curve over $\mathbb{F}$ (geometrically disconnected unless
$m\mid q^2-1$).

The subgroup $G_\zeta(m)\subset G(m)=
\mathrm{GL}_2(\mathbb{Z}/m\mathbb{Z})/\{\pm 1\}$
defined by 
\begin{displaymath}
G_\zeta(m)=\{ A\in G(m)\mid \det(A)\in 
q^{2\mathbb{Z}}\subset (\mathbb{Z}/m\mathbb{Z})^\times\}
\end{displaymath}
acts on both $Y_\zeta(m)$ and $\mathrm{Spec}(\mathbb{F}[\zeta_m])$, 
and the actions
are compatible with the structure map 
$Y_\zeta(m)\map{}\mathrm{Spec}(\mathbb{F}[\zeta_m])$.
Set $G^1(m)=\mathrm{PSL}(\mathbb{Z}/m\mathbb{Z})$, let
\begin{displaymath}
\Gamma_0(m)=
\left\{\left(\begin{matrix} *& *\\ 0& * \end{matrix} \right)\right\}
\subset G(m)
\hspace{1cm}
\Gamma_1(m)=
\left\{\left(\begin{matrix} 1& *\\ 0& * \end{matrix} \right)\right\}
\subset \Gamma_0(m)
\end{displaymath}
be the habitual congruence subgroups,  
and let $\Gamma_\mathrm{Ih}(m)\subset G(m)$ be the center.
For $*\in\{0,1,\mathrm{Ih}\}$ let $Y_*(m)$ be the quotient of
$Y_\zeta(m)$ by the action of 
\begin{displaymath}
\Gamma_*(m)\cap G_\zeta(m).
\end{displaymath}
The function field of the curve $Y_\mathrm{Ih}(m)$ is the field denoted $K_m$
in \cite{ihara}, and there is a canonical isomorphism
of curves over $\mathbb{F}^\mathrm{al}$.
\begin{displaymath}
Y_\mathrm{Ih}(m)\times_{\mathrm{Spec}(\mathbb{F})}\mathrm{Spec}
(\mathbb{F}^\mathrm{al})\cong
Y_\zeta(m)\times_{\mathrm{Spec}(\mathbb{F}[\zeta_m])} 
\mathrm{Spec}(\mathbb{F}^\mathrm{al}).
\end{displaymath}
Denote by $K_*(m)$ the function field of $Y_*(m)$ for 
$*\in\{0,1,\zeta,\mathrm{Ih}\}$ or for $*$ equal to the empty character,
and view these as subfields of some fixed separable closure
$K(1)^\mathrm{sep}$.
Define a $\Lambda$-vector space
$
L_\Lambda=\mathrm{Sym}^{2k-2}\Lambda^2
$
and endow $L_\Lambda$ with an action of 
$\mathrm{Gal}(K(1)^\mathrm{sep}/K(1))$ via
\begin{displaymath}
\mathrm{Gal}(K(1)^\mathrm{sep}/K(1))\map{}\mathrm{Gal}
(K_\zeta(\ell)/K(1))\cong G_\zeta(\ell)
\subset \mathrm{GL}_2(\Lambda)/\{\pm 1\}.
\end{displaymath}

\begin{Rem}\label{ihara factor}
If the action of $\mathrm{GL}_2(\Lambda)$ on $L_\Lambda$ 
is twisted by 
$\det$, then the Galois action is twisted by
the cyclotomic character.  In particular $\Gamma_\mathrm{Ih}(\ell)$
acts trivially on $L_\Lambda\otimes \det^{1-k}$, and 
so the Galois action on $L_\Lambda(1-k)$ factors through 
\begin{displaymath}
\mathrm{Gal}(K(1)^\mathrm{sep}/K(1))\map{}\mathrm{Gal}(K_\mathrm{Ih}(\ell)/K(1)).
\end{displaymath}
Under the bijection between locally constant \'etale sheaves on
$Y_1(N)$ and modules for the absolute Galois group of $K_1(N)$ 
which are unramified outside of the cusps,
$\mathcal{L}_\Lambda(2k-2)$ corresponds to $L_\Lambda$.
\end{Rem}

\begin{Def}
If $M/K(1)$ is a separable extension, 
a \emph{cusp} of $M$ is a place lying
above the place $J=\infty$ of $K(1)$.  A \emph{supersingular prime}
of $M$ is a place lying above a place $J=j$ of $K(1)$ with $j\in\mathbb{F}$
a supersingular $j$-invariant.
\end{Def}

\begin{Thm}\label{ihara theorem}(Ihara)
For any $m>1$ with $(m,q)=1$, $K_\mathrm{Ih}(m)$ has no non-trivial 
everywhere unramified 
extensions in which all supersingular primes are split completely.
Furthermore, $K_\mathrm{Ih}(\infty)=\cup_{(r,q)=1}K_\mathrm{Ih}(r)$ 
is the maximal Galois extension of $K_\mathrm{Ih}(m)$ satisfying
\begin{enumerate}
\item it is tamely ramified, and unramified outside the cusps of 
$K_\mathrm{Ih}(m)$,
\item the supersingular primes of $K_\mathrm{Ih}(m)$ 
are split completely in $K_\mathrm{Ih}(\infty)$.
\end{enumerate}
\end{Thm}

\begin{proof}
This is the main result of \cite{ihara}.
\end{proof}

\begin{Cor}\label{IR cor}
Let $M_\mathrm{Ih}(N)\supset K_\mathrm{Ih}(\infty)$ be the 
maximal separable extension of $K_\mathrm{Ih}(N)$
unramified away from the cusps.
The restriction map on Galois cohomology
\begin{eqnarray}\label{ihara restriction}
\lefteqn{
H^1(M_\mathrm{Ih}(N)/K_\mathrm{Ih}(N), L_\Lambda(1-k))\map{} } \\
\nonumber & &
\left(\bigoplus_v 
H^1(K_\mathrm{Ih}(N)_v , L_\Lambda(1-k))\right)
\bigoplus \left(\bigoplus_w
H^1(K_\mathrm{Ih}(N)_w^\mathrm{unr}, L_\Lambda(1-k)) \right)
\end{eqnarray}
is injective. Here the  sum over $v$ is over all supersingular primes,
the sum over $w$ is  over all cusps, and the superscript $\mathrm{unr}$
denotes maximal unramified extension.
\end{Cor}

\begin{proof}
First consider the restriction map (note Remark \ref{ihara factor})
\begin{equation}\label{Nl restriction}
H^1(M_\mathrm{Ih}(N)/K_\mathrm{Ih}(N\ell), 
L_\Lambda(1-k)) \map{} \bigoplus_{v} 
\mathrm{Hom}(H_v, L_\Lambda(1-k))
\end{equation}
where the sum is over all supersingular primes and all cusps, and
$H_v\subset \mathrm{Gal}(M_\mathrm{Ih}(N)/K_\mathrm{Ih}(N))$ 
is either the decomposition group or inertia group of 
a fixed place of $M_\mathrm{Ih}(N)$ above $v$, 
according as $v$ is supersingular or a cusp.
Any homomorphism from 
$\mathrm{Gal}(M_\mathrm{Ih}(N)/K_\mathrm{Ih}(N\ell))$
to $L_\Lambda(1-k)$ which vanishes on all $H_w$ factors through 
$\mathrm{Gal}(\Phi/K_\mathrm{Ih}(N\ell))$ where $\Phi$ is the maximal Galois
extension of $K_\mathrm{Ih}(N\ell)$ which is everywhere unramified and in
which all supersingular primes split completely.
By Theorem \ref{ihara theorem} $\Phi=K_\mathrm{Ih}(N\ell)$, and so the
map (\ref{Nl restriction}) is injective. Thus any class in the
kernel of (\ref{ihara restriction}) also lies in the kernel of
restriction
\begin{displaymath}
H^1(M_\mathrm{Ih}(N)/K_\mathrm{Ih}(N),L_\Lambda(1-k))\map{}
H^1(M_\mathrm{Ih}(N)/K_\mathrm{Ih}(N\ell),L_\Lambda(1-k)),
\end{displaymath}
and so is in the image of the inflation map
\begin{equation}\label{ihara inflation}
H^1(K_\mathrm{Ih}(N\ell)/K_\mathrm{Ih}(N),L_\Lambda(1-k))\map{} 
H^1(M_\mathrm{Ih}(N)/K_\mathrm{Ih}(N), L_\Lambda(1-k))
\end{equation}
and is unramified at the cusps.
The inertia subgroup in 
$\mathrm{Gal}(K_\mathrm{Ih}(N\ell)/K_\mathrm{Ih}(N))$
of the  cusp $\infty$ is an $\ell$-Sylow subgroup
(this follows from \cite[p. 167]{ihara}), and so any
element in the image of (\ref{ihara inflation}) which is unramified
at the cusps is trivial by \cite[Theorem IX.2.4]{serre:local_fields}.
\end{proof}

\begin{Prop}\label{big ihara}
Let $j:Y_1(N)\hookrightarrow X_1(N)$ be the usual compactification and
assume $\ell\nmid \varphi(N)$. The natural map
\begin{equation}\label{etale H2 cokernel}
H^2_{Z_1(N)}(Y_1(N),\mathcal{L}_\Lambda(k))^\Delta
\map{}H^2(X_1(N),j_*\mathcal{L}_\Lambda(k))^\Delta
\end{equation}
is surjective.
\end{Prop}

\begin{proof}
Let $i:C_1(N)\hookrightarrow X_1(N)$ denote the subscheme of 
cusps, i.e. the complement of $Y_1(N)$ in $X_1(N)$.
From the exact sequence of sheaves on $X_1(N)$
$$
0\map{}j_!\mathcal{L}_\Lambda\map{}j_*\mathcal{L}_\Lambda
\map{}i_*i^*j_*\mathcal{L}_\Lambda\map{}0
$$
and \cite[Proposition II.2.3]{milne:etale} we obtain the exact sequence
$$
H^1(C_1(N), i^*j_*\mathcal{L}_\Lambda(k))\map{}
H^2_c(Y_1(N),\mathcal{L}_\Lambda(k))\map{}
H^2(X_1(N),j_*\mathcal{L}_\Lambda(k))\map{}0
$$
in which the terminating zero is justified by the observation that
closed points on $X_1(N)$, having finite residue field, have cohomological
dimension $1$.  It therefore suffices to prove the surjectivity of
\begin{equation}\label{ihara restriction 0}
H^2_{Z_1(N)}(Y_1(N),\mathcal{L}_\Lambda(k))^\Delta\oplus
H^1(C_1(N), i^*j_*\mathcal{L}_\Lambda(k))^\Delta
\map{}
H^2_c(Y_1(N),\mathcal{L}_\Lambda(k))^\Delta.
\end{equation}
We take $\Lambda$-duals and translate the problem into the
language of Galois cohomology.

Let $M_1(N)$ denote the maximal extension of $K_1(N)$ unramified
outside the cusps.
By \cite[Corollary II.4.13(c)]{milne:duality}
there is an  isomorphism
\begin{displaymath}
H^2_c(Y_1(N),\mathcal{L}_\Lambda(k)) \cong
H^1(M_1(N)/K_1(N),L_\Lambda(1-k))^\vee
\end{displaymath}
in which the superscript $\vee$ denotes $\Lambda$-dual.
If we let $U$ denote the open complement of $Z_1(N)$ in $Y_1(N)$
then the pairing of \cite[Corollary II.3.3]{milne:duality} 
identifies the exact sequence 
\cite[Proposition II.2.3(d)]{milne:duality}
\begin{displaymath}
H^r_c(U,\mathcal{L}_\Lambda(k-1))\rightarrow 
H^r_c(Y_1(N),\mathcal{L}_\Lambda(k-1))
\rightarrow \bigoplus_{z\in Z_1(N)} H^r(z,i_z^*\mathcal{L}_\Lambda(k-1))
\end{displaymath}
with the dual of the relative cohomology sequence
\begin{displaymath}
H^{3-r}(U,\mathcal{L}_\Lambda(k))\leftarrow
H^{3-r}(Y_1(N),\mathcal{L}_\Lambda(k))
\leftarrow H^{3-r}_{Z_1(N)}(Y_1(N),\mathcal{L}_\Lambda(k)).
\end{displaymath}
This gives the first isomorphism of  
\begin{eqnarray*}
H^2_{Z_1(N)}(Y_1(N),\mathcal{L}_\Lambda(k))^\vee &\cong& 
\bigoplus_{z\in Z_1(N)} H^1(z,i_z^*\mathcal{L}_\Lambda(k-1)) \\
& \cong &\hspace{10pt}\bigoplus_{v} H^1(D_v/I_v, L_\Lambda(1-k)) ,
\end{eqnarray*}
in which the second sum is over all supersingular primes and $I_v\subset D_v$
are the inertia and decomposition subgroups in 
$\mathrm{Gal}(M_1(N)/K_1(N))$ of some choice of
place above $v$.
Finally local duality gives the second isomorphism of
\begin{eqnarray*}
H^1(C_1(N), i^*j_*\mathcal{L}_\Lambda(k))^\vee
&\cong&
\bigoplus_w H^1(D_w/I_w, L_\Lambda(2-k)^{I_w})^\vee \\
&\cong&
\bigoplus_w H^1(I_w, L_\Lambda(1-k))^{D_w/I_w}
\end{eqnarray*}
where both sums are over all cusps.
Thus the cokernel of (\ref{ihara restriction 0}) is isomorphic
to the kernel of
\begin{eqnarray}\label{ihara restriction ii}\lefteqn{
H^1(M_1(N)/K_1(N),L_\Lambda(1-k))^\Delta \map{}   }  \\ & &
\left(\bigoplus_v H^1(K_1(N)_v, L_\Lambda(1-k))\right) 
\bigoplus \left(\bigoplus_w H^1(K_1(N)_w^\mathrm{unr}, 
L_\Lambda(1-k))\right)\nonumber
\end{eqnarray}
where again the $v$'s range over supersingular primes and the
$w$ range over cusps.

As we assume that $\ell$ is prime to $\varphi(N)$, the
inflation-restriction sequence identifies the kernel of 
(\ref{ihara restriction ii})  with the kernel of
\begin{eqnarray}\label{ihara restriction iii}\lefteqn{
H^1(M_1(N)/K_0(N),L_\Lambda(1-k)) \map{}   }  \\ & &
\left(\bigoplus_v H^1(K_0(N)_v, L_\Lambda(1-k))\right) 
\bigoplus \left(\bigoplus_w H^1(K_0(N)_w^\mathrm{unr}, 
L_\Lambda(1-k))\right).\nonumber
\end{eqnarray}
The fields $K_1(N)$ and $K_\mathrm{Ih}(N)$ have a common extension
which is unramified outside the cusps (namely $K_\zeta(N)$)
and so $M_1(N)=M_\mathrm{Ih}(N)$.
We may therefore consider the
restriction map
\begin{displaymath}
H^1(M_1(N)/K_0(N),L_\Lambda(1-k))\map{} 
H^1(M_\mathrm{Ih}(N)/K_\mathrm{Ih}(N),L_\Lambda(1-k)),
\end{displaymath}
which is injective as $K_\mathrm{Ih}(N)$ and 
$K_\mathrm{Ih}(\ell)$ are linearly
disjoint over $K(1)$, so that $L_\Lambda(1-k)$ has no
$\mathrm{Gal}(K(1)^\mathrm{sep}/K_\mathrm{Ih}(N))$ invariants.  The kernel of
(\ref{ihara restriction iii}) therefore injects into the kernel of 
(\ref{ihara restriction}), which is trivial by Corollary
\ref{IR cor}.  Thus (\ref{ihara restriction ii}) 
is injective and the proposition is proved.
\end{proof}

The following is our analogue of \cite[Proposition 4.4]{cornut02}.

\begin{Cor}\label{Cor:Ihara application}
Assume $\ell\nmid \varphi(N)$.  The $\Lambda$-augmented Kummer map
\begin{displaymath}
\mathcal{A}_\Lambda^\circ (Z_1(N);Y_1(N))^\Delta\map{}
H^1\big(\mathbb{F}^\mathrm{al}/\mathbb{F}, 
\tilde{H}^1(Y_1(N)_{/\mathbb{F}^\mathrm{al}},
\mathcal{L}_\Lambda)  (k)\big)^\Delta
\end{displaymath}
of Definition  \ref{kummer} is surjective.
\end{Cor}

\begin{proof}
Lemma \ref{purity} gives isomorphisms
\begin{displaymath}
\mathcal{A}_\Lambda^\circ (Z_1(N);Y_1(N))\cong 
\bigoplus_{z\in Z_1(N)} H^2_z(Y_1(N),\mathcal{L}_\Lambda(k))
\cong H^2_{Z_1(N)}(Y_1(N),\mathcal{L}_\Lambda(k))
\end{displaymath}
which restrict to isomorphisms of $\Delta$-invariants.
The claim is now immediate from Lemma \ref{HS} and
Proposition \ref{big ihara}.
\end{proof}

%%%%%%%%%%%%%%%%%%%%%%%%%%%%%%%%%%%%%%%%%%%%%%%%%%%%%%%%%%%%

\subsection{Degeneracy maps on supersingular points}

%%%%%%%%%%%%%%%%%%%%%%%%%%%%%%%%%%%%%%%%%%%%%%%%%%%%%%%%%%%%

Suppose $M=rM'$ for a prime $r$.  The following theorem and its proof
are based on work of Ribet \cite[Theorem 3.15]{ribet}.

\begin{Prop}\label{degen appendix}
Assume $\ell\nmid\varphi(N)$ and $\ell>2k-2$, and abbreviate
\begin{displaymath}
\mathcal{Z}(M)=\mathcal{A}_\Lambda^\circ(Z_1(N,M);Y_1(N,M))^\Delta
\end{displaymath}
and similarly for $M'$. The sum of the degeneracy maps of \S \ref{degen}
\begin{displaymath}
\alpha^M_{M'}\oplus\beta^M_{M'}:
\mathcal{Z}(M)\map{}\mathcal{Z}(M')\oplus\mathcal{Z}(M')
\end{displaymath}
is surjective.
\end{Prop}

\begin{proof}
Let $\mathrm{N}_\Delta$ be the norm element in the group algebra
$\Lambda[\Delta]$.
Suppose we are given a $\Lambda$-augmented $\Gamma_1(N,M')$
structure over $\mathbb{F}^\mathrm{al}$
\begin{displaymath}
(E,x,\Theta)\in \mathcal{A}_\Lambda(\Gamma_1(N,M'))
\end{displaymath}
with $E$ supersingular, and a degree $r^{2n}$ endomorphism $f:E\map{}E$ 
preserving the $\Gamma_0(NM')$ structure underlying $x$.  
Factor the endomorphism $f:E\map{}E$
as 
\begin{displaymath}
E=E_0\map{h_1}E_1\map{h_2}\cdots\map{h_{2n-1}}E_{2n-1}
\map{h_{2n}}E_{2n}=E
\end{displaymath}
with each $h_i$ of degree $r$.  Set $f_i=h_i\circ\cdots\circ h_1:E\map{}E_i$
and let $x_i=f_i(x)$ be the induced $\Gamma_1(N,M')$ structure
on $E_i$.  For $i<2n$ let $y_i$ be the $\Gamma_1(N,M)$ structure on $E_i$
obtained by adding the $\Gamma_0(r)$ structure $\mathrm{ker}(h_{i+1})$
to $x_i$, and for $i>0$ let $y_i^\vee$ be the $\Gamma_1(N,M)$ structure
obtained by adding the $\Gamma_0(r)$ structure $\mathrm{ker}(h_i^\vee)$.
Define 
\begin{displaymath}
\Theta_i=r^{i(1-k)}f_i(\Theta)\in\mathcal{A}_\Lambda(E_i).
\end{displaymath}
A simple calculation of the degeneracy maps
of \S \ref{degen} 
%\begin{eqnarray*}
%\mathrm{N}_\Delta\cdot \alpha^M_{M'}(E_i,y_i,\Theta_i) &=&
%\mathrm{N}_\Delta\cdot (E_i,x_i,\Theta_i)\\
%\mathrm{N}_\Delta\cdot \beta^M_{M'}(E_i,y_i,\Theta_i) &=&
%r^{k-1}\mathrm{N}_\Delta\cdot (E_{i+1},x_{i+1},\Theta_{i+1})\\
%\mathrm{N}_\Delta\cdot \alpha^M_{M'}(E_i,y_i^\vee,\Theta_i) &=&
%\mathrm{N}_\Delta\cdot (E_i,x_i,\Theta_i)\\
%\mathrm{N}_\Delta\cdot \beta^M_{M'}(E_i,y_i^\vee,\Theta_i) &=&
%r^{k-1}\mathrm{N}_\Delta\cdot (E_{i-1},x_{i-1},\Theta_{i-1}).
%\end{eqnarray*}
shows that the element 
\begin{displaymath}
T = T_{E,x,f,\Theta} \in \mathcal{A}_\Lambda(\Gamma_1(N,M))
\end{displaymath}
defined by
\begin{eqnarray*}
T &=&
\mathrm{N}_\Delta\big[(E_0,y_0,\Theta_0)-(E_2,y_2^\vee,\Theta_2)
+ (E_2,y_2,\Theta_2)-(E_4,y_4^\vee,\Theta_4) + \\
& &
\ldots + 
(E_{2n-2},y_{2n-2},\Theta_{2n-2})-(E_{2n},y_{2n}^\vee,\Theta_{2n})\big]
\end{eqnarray*}
satisfies $\beta^M_{M'}(T)= 0$ and
\begin{displaymath}
\alpha^M_{M'}(T) =
\mathrm{N}_\Delta\cdot \big(E,x,\Theta-r^{2n(1-k)}f(\Theta)\big).
\end{displaymath}
It follows from Lemma \ref{local moduli} 
(with $N$ replaced by $\mathbf{N}=NM$)
that $T$ is fixed by the action of 
$\mathrm{Gal}(\mathbb{F}^\mathrm{al}/\mathbb{F})$, and so
defines an element of $\mathcal{Z}(M)$.

We pause for a 

\begin{Lem}\label{irreducible appendix}
With $(E,x)$ as above, let $D$ denote the 
$\Gamma_0(NM')$ structure underlying the $\Gamma_1(N,M')$
structure $x$. The $\Lambda$-module $\mathcal{A}_\Lambda(E)$ has a set
of generators $A_{E,x}$ such that each $a\in A_{E,x}$ has the form
$a=\Theta_a-\deg(f_a)^{(1-k)}f_a(\Theta_a)$
for some $\Theta_a\in\mathcal{A}_\Lambda(E)$
and some endomorphism  $f_a:E\map{}E$ such that
$f_a(D)=D$ and $\deg(f_a)$ is an even power of $r$.
\end{Lem}

\begin{proof}
Set $R=\mathrm{End}_{\mathbb{F}^\mathrm{al}}(E,D)$, a level 
$\mathrm{N}$ Eichler order in 
a quaternion algebra ramified exactly at $q$ and $\infty$, and 
let $\Gamma=R[1/r]^\times$.  Let $\rho$ denote the 
natural action of $R$ on $\mathcal{A}_\Lambda(E)$, extend $\rho$ to an action
of $\Gamma$ (recall $\ell\nmid M$ so that $r\not=\ell$), and 
let $\rho^*=\rho\otimes\det^{1-k}$ be the twist 
such that $\mathbb{Z}[1/r]^\times\subset\Gamma$ acts trivially.  All of this 
notation is exactly as in \S \ref{reduction} with $p$ replaced by $r$.
As in the proof of Lemma \ref{irreducible}, $\mathcal{A}_\Lambda(E)$
has no submodules stable under the restriction of $\rho^*$ to the
subgroup of norm one elements $\Gamma^1\subset \Gamma$.
As the set 
\begin{displaymath}
A_{E,x}=\{ \Theta-\rho^*(\gamma)\Theta\mid  \Theta\in\mathcal{A}_\Lambda(E),
\gamma\in\Gamma^1\}
\end{displaymath}
is stable under the action of $\rho^*(\Gamma^1)$, it must generate
$\mathcal{A}_\Lambda(E)$.
For each 
\begin{displaymath}
\Theta-\rho^*(\gamma)\Theta\in A_{E,x},
\end{displaymath}
let $f=r^n\gamma$ for $n$ large enough that
$r^n\gamma\in R$.  Then $f$ has degree $r^{2n}$ and 
\begin{displaymath}
\Theta-\deg(f)^{1-k}f(\Theta)=\Theta-\rho^*(f)\Theta=
\Theta-\rho^*(\gamma)\Theta,
\end{displaymath}
so that $A_{E,x}$ has the desired properties.
\end{proof}

If we let $E$ vary over all supersingular elliptic curves over 
$\mathbb{F}^\mathrm{al}$,
$x$ vary over all $\Gamma_1(N,M)$ structures on $E$, and
$\Theta'$ vary over the set $A_{E,x}$ of
 Lemma \ref{irreducible appendix}, the elements
\begin{displaymath}
\mathrm{N}_\Delta\cdot (E,x,\Theta')
\in\mathcal{A}_\Lambda(\Gamma_1(N,M'))
\end{displaymath} generate the submodule $\mathcal{Z}(M')$.
Hence, by the construction of $T_{E,x,f,\Theta}$ above, there is a family
$\{T_i\}\subset \mathcal{Z}(M)$ such that
$\beta^M_{M'}(T_i)=0$ for all $i$ and such that 
$\{\alpha^M_{M'}(T_i)\}$ generates $\mathcal{Z}(M')$.  
A construction similar to that of $T$ produces a 
family with the same properties but with the roles of
$\alpha$ and $\beta$ reversed, completing the proof of Proposition 
\ref{degen appendix}.
\end{proof}

%%%%%%%%%%%%%%%%%%%%%%%%%%%%%%%%%%%%%%%%%%%%%%%%%%%%%%%%%%%%%%%%%%%%%
%%%%%%%%%%%%%%%%%%%%%%%%%%%%%%%%%%%%%%%%%%%%%%%%%%%%%%%%%%%%%%%%%%%%%

\section{Nonvanishing of Heegner classes}
\label{S: money}

%%%%%%%%%%%%%%%%%%%%%%%%%%%%%%%%%%%%%%%%%%%%%%%%%%%%%%%%%%%%%%%%%%%%%
%%%%%%%%%%%%%%%%%%%%%%%%%%%%%%%%%%%%%%%%%%%%%%%%%%%%%%%%%%%%%%%%%%%%%

Keep $K$, $E_1$, $\mathfrak{N}$, and $\mathrm{Ta}_p(E_1)\cong\mathbb{Z}_p^2$
as in \S \ref{tree}, so that $K$ is an imaginary quadratic field
in which the prime divisors of $N$ are split,
$\mathcal{O}_K/\mathfrak{N}\cong \mathbb{Z}/N\mathbb{Z}$, and $E_1$ is an elliptic curve
over $\mathbb{Q}^\mathrm{al}$ with complex multiplication by $\mathcal{O}_K$.
Let $D=\mathrm{disc}(K)$ and let $H[p^s]$, $\mathcal{G}$, and $G_0$
be as in \S \ref{intro}. Let $f\in S_{2k}(\Gamma_0(N),\mathbb{C})$,
$\Phi$, $\chi$, and $\pi_\chi$ also be as in \S \ref{intro}.
Let $\mathbf{T}$ be the $\mathbb{Z}$-algebra generated by the Hecke operators
$\{ T_m\mid (m,N)=1 \}$ and the group of diamond operators $\Delta$
acting on $S_{2k}(\Gamma_1(N),\mathbb{C})$,
so that $f$ determines an idempotent $\pi_f$ in the semi-simple
$\Phi$-algebra $\mathbf{T}\otimes_\mathbb{Z} \Phi$.

%%%%%%%%%%%%%%%%%%%%%%%%%%%%%%%%%%%%%%%%%%%%%%%%%%%%%%%%%%%%%%%%%%

\subsection{Heegner cohomology classes}
\label{S: Heegner cohomology}

%%%%%%%%%%%%%%%%%%%%%%%%%%%%%%%%%%%%%%%%%%%%%%%%%%%%%%%%%%%%%%%%%%

For each finite quotient $\Lambda$ of $\mathbb{Z}_\ell$ we have the
$\mathrm{Gal}(\mathbb{Q}^\mathrm{al}/K)$-module $W_\Lambda$ of 
(\ref{our module})
and, for each $s\ge 0$,  
the family of cohomology classes $\Omega_s(g)$ of (\ref{class param})
parametrized by $g\in \mathcal{T}_s\subset\mathcal{T}$.
If we set $W_{\mathbb{Z}_\ell}=\varprojlim W_{\mathbb{Z}/\ell^e\mathbb{Z}}$, 
then the classes
$\Omega_s(g)$ are compatible as $\Lambda=\mathbb{Z}/\ell^e\mathbb{Z}$ varies, and define 
classes
\begin{displaymath}
\Omega_s(g)\in H^1(\mathbb{Q}^\mathrm{al}/H[p^s], W_{\mathbb{Z}_\ell}(k))^\Delta,
\end{displaymath}
and also classes (denoted the same way) in the cohomology of 
$W_\Lambda=W_{\mathbb{Z}_\ell}\otimes\Lambda$
for any $\mathbb{Z}_\ell$-algebra $\Lambda$.  Other constructions made with 
$\Lambda=\mathbb{Z}/\ell^e\mathbb{Z}$ extend to any $\mathbb{Z}_\ell$-algebra $\Lambda$
in the same way. We denote by
\begin{displaymath}
\mathrm{Heeg}_s\subset  H^1\big(\mathbb{Q}^\mathrm{al}/H[p^s],W_\Phi(k)\big)^\Delta
\end{displaymath}
the $\Phi$-submodule generated by the classes $\Omega_s(g)$ 
as $g$ ranges over $\mathcal{T}_s$.
By a well known theorem of Deligne, the Hecke algebra 
$\mathbf{T}\otimes_\mathbb{Z} \Phi$
acts on $W_\Phi$, and the Galois representation 
$W_f=\pi_f W_\Phi$ is a two dimensional $\Phi$-vector space.
Set 
\begin{displaymath}
\mathrm{Heeg}_s(f)=\pi_f \mathrm{Heeg}_s.
\end{displaymath}

\begin{Thm}\label{big money}
Fix a character $\chi:G_0\map{}\Phi^\times$ and let $\pi_\chi$ be as in
\S \ref{intro}.
Suppose $\ell$ does not divide $p$, $N$, $\varphi(N)$, $\mathrm{disc}(K)$,
or $(2k-2)!$
As $s$ grows the $\Phi$-dimension of 
$\pi_\chi\mathrm{Heeg}_s(f)$
grows without bound.
\end{Thm}

\begin{proof}
Let $r_1, r_2, \ldots $ be the prime
divisors of $D$ and let $\mathfrak{r}_i$ denote the unique prime
of $K$ above $r_i$.  Let $G_1\subset \mathcal{G}$ be the 
subgroup generated by the Frobenius classes of the $\mathfrak{r}_i$,
so that $G_1$ has exponent $2$, and in particular $G_1\subset G_0$.
Reordering the $r_i$ if needed, choose $n$ such that the 
Frobenius classes of $\mathfrak{r}_1,\ldots, \mathfrak{r}_n$
form a basis for the $\mathbb{Z}/2\mathbb{Z}$-vector space $G_1$.  
Set $M=r_1\cdots r_n$, so that divisors of $M$
are naturally in bijection with the elements of $G_1$.  We denote this
bijection by $d\mapsto \sigma_d$.
Set $\mathfrak{M}=\mathfrak{r}_1\cdot\ldots\cdot \mathfrak{r}_k$.  
For each $\sigma\in G_1$ fix once and for all an extension of
$\sigma$ to $\mathrm{Gal}(\mathbb{Q}^\mathrm{al}/K)$, 
and let $\mathcal{S}_1$ denote
the set of extensions so chosen.  Let 
$\mathcal{S}_0\subset\mathrm{Gal}(\mathbb{Q}^\mathrm{al}/K)$
be chosen so that restriction to $H[p^\infty]$ takes 
$\mathcal{S}_0$ injectively into $G_0$ with image equal to a 
set of representatives for the cosets $G_0/G_1$.  
Let $\mathcal{S}=\{ \sigma\tau\mid \sigma\in \mathcal{S}_1,
\tau\in\mathcal{S}_0\}$.

As in \S \ref{S: reduction of the family}, let $\mathfrak{Q}$
be a finite set of rational primes, all inert in $K$ and all
prime to $\ell p \mathbf{N}$, and fix extensions of these places
to $\mathbb{Q}^\mathrm{al}$.  We will continue our practice of
writing $\mathfrak{q}\in\mathfrak{Q}$ to indicate that 
$\mathfrak{q}$ is the prime
of $K$ above the rational prime $q\in\mathfrak{Q}$.
For each $\mathfrak{q}\in\mathfrak{Q}$ define, 
using the notation (\ref{abbreviation}),
$\lambda_d:\mathcal{Z}_\mathfrak{q}(M)\map{}\mathcal{Z}_\mathfrak{q}(1)$
by $\lambda_d=d^{1-k}\cdot(\beta^d_1\circ \alpha^M_d)$
where $\alpha_d^M$ and $\beta_1^d$ are the degeneracy maps of \S \ref{degen}.
Consider the composition
\begin{equation}\label{SRMII}
(\mathbb{Z}/\ell\mathbb{Z})[\mathcal{T}]\map{}
\bigoplus_{(\sigma,\mathfrak{q})\in\mathcal{S}_0\times\mathfrak{Q}}
\mathcal{Z}_\mathfrak{q}(M)
\map{\oplus_{d\mid M} \lambda_d}
\bigoplus_{(\sigma,\mathfrak{q})\in\mathcal{S}_0\times\mathfrak{Q}}
\bigoplus_{d\mid M} \mathcal{Z}_\mathfrak{q}(1)
\map{}
\bigoplus_{(\sigma,\mathfrak{q})\in\mathcal{S}\times\mathfrak{Q}} 
\mathcal{Z}_\mathfrak{q}(1)
\end{equation}
in which the first arrow is the map $\mathrm{Red}_{\mathcal{S}_0,\mathfrak{Q}}$
of (\ref{SRM}), and the final arrow rearranges the sum, taking 
the summand $(\sigma,\mathfrak{q},d)$
to the summand $(\sigma_d \sigma,\mathfrak{q})$.

\begin{Lem}\label{main surjectivity}
The composition (\ref{SRMII}) is surjective, and is equal to the 
simultaneous reduction map (\ref{SRM}) defined with $M=1$.  
\end{Lem}

\begin{proof}
By \cite[Lemma 4.5]{cornut02} the set $\mathcal{S}_0$ is chaotic
in the sense of Definition \ref{chaotic}, and so Theorem \ref{chaos}
gives the surjectivity of $\mathrm{Red}_{\mathcal{S}_0,\mathfrak{Q}}$.
The surjectivity of $\oplus_{d|M}\lambda_d$ is an easy induction 
using Proposition \ref{degen appendix}.

Fix $g\in\mathcal{T}$ and let $A$ be an elliptic curve over $\mathbb{Q}^\mathrm{al}$ with 
complex multiplication by $\mathcal{O}_g\subset K$. If $d|M$, let 
$\mathfrak{d}$ be the unique $\mathcal{O}_g$-ideal
of norm $d$ and set $A'=A/A[\mathfrak{d}]$. The main theorem
of complex multiplication provides an isomorphism 
$A'\cong A^{\sigma_d}$ such that the composition
$A\map{}A'\cong A^{\sigma_d}$ agrees with $P\mapsto P^{\sigma_d}$
for all torsion points $P\in A(\mathbb{Q}^\mathrm{al})$ of order prime to $d$.
Thus
\begin{displaymath}
\lambda_d(E_g^\sigma,\mathbf{C}_g^\sigma,\Theta_g^\sigma)
= (E_g^{\sigma}, C_g^\sigma,\Theta_g^\sigma)^{\sigma_d}
\end{displaymath}
for any $\sigma\in\mathrm{Gal}(\mathbb{Q}^\mathrm{al}/K)$, and the lemma follows.
\end{proof}

For each $(\sigma,\mathfrak{q})\in\mathcal{S}\times\mathfrak{Q}$
and each $g\in\mathcal{T}_s$ the cohomology class $\Omega_s(g)$
is unramified at $\mathfrak{q}$, and, since the residue field of 
$H[p^s]$ at $\mathfrak{q}$ is $\mathbb{F}_\mathfrak{q}$, the localization of 
$\Omega_s(g)$ at $\mathfrak{q}$ defines   a class
\begin{displaymath}
\mathrm{loc}_{\sigma,\mathfrak{q}}(g)\in 
H^1(\mathbb{F}_\mathfrak{q}^\mathrm{al}/\mathbb{F}_\mathfrak{q},
W_{\mathbb{Z}_\ell}(k))^\Delta.
\end{displaymath}
Summing over all $(\sigma,\mathfrak{q})\in \mathcal{S}\times\mathfrak{Q}$
and extending linearly to the free $\mathbb{Z}_\ell$-module on $\mathcal{T}_s$
defines
\begin{displaymath}
\mathrm{loc}_{\mathcal{S},\mathfrak{Q}}: \mathbb{Z}_\ell[\mathcal{T}_s]
\map{}\bigoplus_{(\sigma,\mathfrak{q})\in \mathcal{S}\times\mathfrak{Q}}
H^1(\mathbb{F}_\mathfrak{q}^\mathrm{al}/\mathbb{F}_\mathfrak{q},
W_{\mathbb{Z}_\ell}(k))^\Delta.
\end{displaymath}
This map is compatible with the natural inclusions as $s$ varies. 
Proposition \ref{Prop:reduction} 
gives the commutative diagram
\begin{equation}\label{reduction diagram}
\xymatrix{
{\mathbb{Z}_\ell[\mathcal{T}_s]}
\ar[d]_{\mathrm{Red}_{\mathcal{S},\mathfrak{Q}}} 
\ar[r]^{\mathrm{loc}_{\mathcal{S},\mathfrak{Q}}} &
{\bigoplus
H^1(\mathbb{F}_\mathfrak{q}^\mathrm{al}/\mathbb{F}_\mathfrak{q},
W_{\mathbb{Z}_\ell}(k))^\Delta}\ar[d]\\
{\bigoplus
\mathcal{A}_{\mathbb{Z}/\ell\mathbb{Z}}^\circ(Z_1(N)_\mathfrak{q}; 
Y_1(N)_{/\mathbb{F}_\mathfrak{q}})^\Delta} \ar[r]
& 
{\bigoplus
H^1(\mathbb{F}_\mathfrak{q}^\mathrm{al}/\mathbb{F}_\mathfrak{q},
W_{\mathbb{Z}/\ell\mathbb{Z}}(k))^\Delta}
}\end{equation}
where all sums are over $\mathcal{S}\times\mathfrak{Q}$,
$\mathrm{Red}_{\mathcal{S},\mathfrak{Q}}$ is the restriction of
the simultaneous reduction map (\ref{SRM}), with $M=1$, to 
$\mathbb{Z}_\ell[\mathcal{T}_s]$, 
and the bottom horizontal arrow 
is the $\mathbb{Z}/\ell\mathbb{Z}$-augmented Kummer 
map of Definition \ref{kummer}.
By Corollary \ref{Cor:Ihara application} the bottom horizontal arrow is 
surjective, and by Lemma \ref{main surjectivity} the restriction of
$\mathrm{Red}_{\mathcal{S},\mathfrak{Q}}$ to 
$\mathbb{Z}_\ell[\mathcal{T}_s]$ is surjective for
$s\gg 0$.
The same argument as \cite[Lemma 2.2]{nek:euler} gives the exactness of
\begin{displaymath}
0\map{} W_{\mathbb{Z}_\ell}(k)\map{\ell}W_{\mathbb{Z}_\ell}(k)\map{}
W_{\mathbb{Z}/\ell\mathbb{Z}}(k)\map{}0,
\end{displaymath}
and taking $\mathbb{F}_\mathfrak{q}^\mathrm{al}/\mathbb{F}_\mathfrak{q}$ 
cohomology shows that the 
right vertical arrow is surjective with kernel equal to the image
of multiplication by $\ell$.  Applying Nakayama's lemma, we have proved

\begin{Lem}\label{final reduction I}
For $s\gg 0$ the restriction of 
$\mathrm{loc}_{\mathcal{S},\mathfrak{Q}}$ to $\mathbb{Z}_\ell[\mathcal{T}_s]$ is surjective.
\end{Lem}

Let $R$ be the integer ring of $\Phi$, so that $W_R$ is
an $R$ lattice in $W_\Phi$ and
$\pi_f W_R$ is an $R$ lattice in $W_f=\pi_f W_\Phi$. Let 
\begin{displaymath}
\mathrm{Heeg}_{R,s} \subset H^1(\mathbb{Q}^\mathrm{al}/H[p^s], W_R(k))^\Delta
\end{displaymath} 
be the $R$ submodule generated by 
the classes $\Omega_s(g)$ for $g\in\mathcal{T}_s$,
and abbreviate 
\begin{displaymath}
T=\pi_f W_R(k)\subset W_f(k).
\end{displaymath}

\begin{Lem}\label{final reduction II}
For $s\gg 0$, the image of the composition 
\begin{displaymath}
\mathrm{Heeg}_{R,s}\map{\pi_\chi\pi_f}  
H^1(\mathbb{Q}^\mathrm{al}/H[p^s], T) \map{\oplus\mathrm{loc}_\mathfrak{q}}
\bigoplus_{\mathfrak{q}\in\mathfrak{Q}}
H^1(\mathbb{Q}^\mathrm{al}_\mathfrak{q}/K_\mathfrak{q}, T)
\end{displaymath}
is $\bigoplus_{\mathfrak{q}\in\mathfrak{Q}} 
H^1(\mathbb{F}^\mathrm{al}_\mathfrak{q}/\mathbb{F}_\mathfrak{q},T)$,
the submodule of unramified cohomology classes.
\end{Lem}

\begin{proof}
Using Proposition \ref{Prop:reduction} we see that the image of the
composition lies in the unramified cohomology, and is equal to the
image of 
\begin{eqnarray*}\lefteqn{
R[\mathcal{T}_s]\map{\mathrm{loc}_{\mathcal{S},\mathfrak{Q}}}
\bigoplus_{\mathcal{S}}\bigoplus_{\mathfrak{q}\in \mathfrak{Q}} 
H^1(\mathbb{F}_\mathfrak{q}^\mathrm{al}/
\mathbb{F}_\mathfrak{q}, W_R(k))^\Delta } \hspace{2cm}\\
& &\map{\chi}
\bigoplus_{\mathfrak{q}\in \mathfrak{Q}} 
H^1(\mathbb{F}_\mathfrak{q}^\mathrm{al}/\mathbb{F}_\mathfrak{q}, W_R(k))^\Delta
\map{\pi_f}
\bigoplus_{\mathfrak{q}\in \mathfrak{Q}} 
H^1(\mathbb{F}_\mathfrak{q}^\mathrm{al}/\mathbb{F}_\mathfrak{q}, T)
\end{eqnarray*}
where arrow labeled $\chi$ takes  the element 
$(x_\sigma)_{\sigma\in\mathcal{S}}$ to 
$\sum_{\sigma\in\mathcal{S}}\chi(\sigma)x_\sigma$.
The first arrow is surjective for $s\gg 0$ 
by Lemma \ref{final reduction I},
the second is obviously surjective, and the third is surjective
by the fact that 
$\mathrm{Gal}(\mathbb{F}^\mathrm{al}_\mathfrak{q}/\mathbb{F}_\mathfrak{q})$ 
has cohomological dimension one.
\end{proof}

Let $\mathfrak{m}$ denote the maximal ideal of $R$ and 
set $\overline{T}=T\otimes_R R/\mathfrak{m}$.
If $q\nmid\ell N D$ is a rational prime whose 
absolute Frobenius acts as  complex conjugation on $K(\overline{T})$,
the extension of $\mathbb{Q}$ cut out by the Galois action on
$\overline{T}=T\otimes_R R/\mathfrak{m}$,
then clearly $q$ is inert in $K$ and the Frobenius of the unique
prime $\mathfrak{q}$ of $K$ above $q$ acts trivially on $\overline{T}$.
By the Chebetarov theorem we may choose $\mathfrak{Q}$ as large as we want
and containing only primes of this form.
For $s\gg 0$, Lemma \ref{final reduction II} gives a surjection from 
$\pi_\chi\pi_f\mathrm{Heeg}_{R,s}$ to
\begin{displaymath}
\bigoplus_{\mathfrak{q}\in\mathfrak{Q}} 
H^1(\mathbb{F}_\mathfrak{q}^\mathrm{al}/\mathbb{F}_\mathfrak{q},T)\otimes R/\mathfrak{m}
\cong H^1(\mathbb{F}_\mathfrak{q}^\mathrm{al}/\mathbb{F}_\mathfrak{q},\overline{T})
\cong \bigoplus_{\mathfrak{q}\in\mathfrak{Q}} 
\overline{T}/(\mathrm{Frob}_\mathfrak{q}-1)\overline{T}
\cong \bigoplus_{\mathfrak{q}\in\mathfrak{Q}}  \overline{T}.
\end{displaymath}
Thus the $R/\mathfrak{m}$ dimension of
$(\pi_\chi\pi_f\mathrm{Heeg}_{R,s})\otimes_R R/\mathfrak{m}$
is at least $\#\mathfrak{Q}$ for $s\gg 0$.  Enlarging 
$\mathfrak{Q}$, the $R/\mathfrak{m}$ dimension of
$(\pi_\chi\pi_f\mathrm{Heeg}_{R,s})\otimes_R R/\mathfrak{m}$
grows without bound as $s$ increases.

\begin{Lem}\label{finite torsion}
The $R$-torsion submodule of $H^1(\mathbb{Q}^\mathrm{al}/H[p^s],T)$
is finite and of bounded order as $s\to\infty$.
\end{Lem}

\begin{proof}
The $R$-torsion submodule of $H^1(\mathbb{Q}^\mathrm{al}/H[p^s],T)$ 
is isomorphic to the quotient of
\begin{equation}\label{an H0}
H^0(\mathbb{Q}^\mathrm{al}/H[p^s],T\otimes_{\mathbb{Z}_\ell}
(\mathbb{Q}_\ell/\mathbb{Z}_\ell))
\end{equation}
by its maximal divisible subgroup.  
Let $\ell\nmid pN$ be a rational prime which is inert in $K$, and
let $\lambda$ be the prime of $K$ above $\ell$.  Then $\lambda$
splits completely in $H[p^\infty]$ and, by Deligne's proof of the
Ramanujan conjecture, $\mathrm{Frob}_\lambda=\mathrm{Frob}_\ell^2$ 
acts on $W_f(k)$ with  eigenvalues of (complex) absolute value 
$\ell^{2k-1}$.  Hence $\mathrm{Frob}_\lambda-1$ is invertible on 
$W_f(k)\cong T\otimes_{\mathbb{Z}_\ell}\mathbb{Q}_\ell$. A 
snake lemma argument then shows that the order of (\ref{an H0})
is bounded by the order of $T/(\mathrm{Frob}_\lambda-1)T$.
\end{proof}

The $R$-torsion submodule of 
$\pi_\chi\pi_f\mathrm{Heeg}_{R,s}$ is contained in the torsion
submodule of $H^1(\mathbb{Q}^\mathrm{al}/H[p^s],T)$, and so is finite
and bounded as $s\to\infty$ by Lemma \ref{finite torsion}.
We have seen that the $R/\mathfrak{m}$ dimension
of $(\pi_\chi\pi_f\mathrm{Heeg}_{R,s})\otimes_R R/\mathfrak{m}$
increases without bound, and it now follows that the 
$R$-rank of $\pi_\chi\pi_f\mathrm{Heeg}_{R,s}$ also increases without bound.
This complete the proof of Theorem \ref{big money}.
\end{proof}

\bibliographystyle{plain}
%\bibliography{cycles.bib}

\end{document}